\documentclass[11pt, letterpaper]{amsart}
\usepackage[left=1in,right=1in,bottom=0.5in,top=0.8in]{geometry}
\usepackage{amsfonts}
\usepackage{amsmath, amssymb}
\usepackage{graphicx}
\usepackage[font=small,labelfont=bf]{caption}
\usepackage{epstopdf}
\usepackage[pdfpagelabels,hyperindex]{hyperref}
\usepackage{xcolor}
\usepackage{amsthm}
\usepackage{float}
\usepackage{pgfplots}
\usepackage{listings}
\usepackage{longtable}
\usepackage{mathrsfs}

\newcommand{\overbar}[1]{\mkern 1.5mu\overline{\mkern-1.5mu#1\mkern-1.5mu}\mkern 1.5mu}

\newtheorem{thm}{Theorem}[section]

\newtheorem{prop}[thm]{Proposition}
\newtheorem{lem}[thm]{Lemma}

\theoremstyle{definition}
\newtheorem{defn}[thm]{Definition}

\theoremstyle{remark}
\newtheorem{rem}[thm]{Remark}

\pgfplotsset{compat=1.9}

\newcommand{\RN}[1]{%
  \textup{\uppercase\expandafter{\romannumeral#1}}%
}
\newcounter{x}

\numberwithin{equation}{section}

\author[Junfu Yao]{Junfu Yao}
\email{jyao21@jhu.edu}

\title{Relative Entropy of Hypersurfaces in Hyperbolic Space}
\begin{document}
\begin{abstract}
We study a notion of relative entropy for certain hypersurfaces in hyperbolic space. We relate this quantity to the renormalized area introduced by Graham-Witten\cite{GrahamWitten}. We also obtain a monotonicity formula for relative entropy applied to mean curvature flows in hyperbolic space.
\end{abstract}
\maketitle

\section{Introduction}
Let $n\geq2$ and $\mathbb{H}^{n+1}$ denote the ($n+1$)-dimensional hyperbolic space. We recall that the Poincar\'e half-space model for $\mathbb{H}^{n+1}$ is the open upper half-space
\begin{equation*}
    \mathbb{R}^{n+1}_+=\{(\textbf{x},y)\in\mathbb{R}^n\times\mathbb{R}:\ y>0\},
\end{equation*}
equipped with the metric
\begin{equation*}
    g_P=\frac{1}{y^2}(d\textbf{x}\otimes d\textbf{x}+dy\otimes dy)=\frac{1}{y^2}g_{\mathbb{R}^{n+1}}.
\end{equation*}
In this model, we have the standard compatification: $\overbar{\mathbb{H}}^{n+1}=\mathbb{H}^{n+1}\cup(\mathbb{R}^n\times\{0\})\cup\{\infty\}$. The set $(\mathbb{R}^n\times\{0\})\cup\{\infty\}=\mathbb{S}^{n+1}$ is called the \textit{ideal boundary} of $\mathbb{H}^{n+1}$ and denoted by $\partial_\infty\mathbb{H}^{n+1}$. The ideal boundary of $\mathbb{H}^{n+1}$ is associated to a unique conformal structure. A \textit{boundary defining function} $r$ on $\overbar{\mathbb{H}}^{n+1}$ is a function such that the metric $r^2g_P$ extends smoothly to $\overbar{\mathbb{H}}^{n+1}$.

Minimal hypersurfaces in $\mathbb{H}^{n+1}$ are critical points of the functional
\begin{equation}\label{31}
    \text{Vol}(\Sigma)=\int_{\Sigma}1\cdot d\mathcal{H}^n_{\mathbb{H}^{n+1}}=\int_{\Sigma}\frac{1}{y^n}\cdot d\mathcal{H}^n.
\end{equation}
Minimal hypersurfaces in hyperbolic space have been extensively studied, see \cite{Anderson1982}, \cite{Anderson1983}, \cite{Gabai97}, \cite{deOliveiraSoret98}, \cite{MartinWhite2013}, \cite{HardtLin}, \cite{FLin}, \cite{TonegawaCMCinHpSpace}, \cite{Coskunuzer2006Generic}, \cite{Coskunuzer2011Number}, \cite{AlexakisMazzeo} and \cite{HuangWang2015Counting}.

As one may observe, there are no closed minimal hypersurfaces in $\mathbb{H}^{n+1}$. So, it is natural to consider minimal hypersurfaces that are asymptotic to the ideal boundary (See Definition \ref{66} for the terminology). However, the volume (\ref{31}) is always divergent for any $\Sigma$ that has asymptotic boundary. Let $r$ be any boundary defining function for $\overbar{\mathbb{H}}^{n+1}$. In \cite{GrahamWitten}, Graham-Witten showed that for a sufficiently regular minimal hypersurface $\Sigma$ and an even integer $n$, the volume of $\{r\geq\epsilon\}\cap\Sigma$ has an expansion
\begin{equation}\label{36}
    \text{Vol}_\epsilon(\Sigma;r)=\int_{\Sigma\cap\{r\geq\epsilon\}}1\cdot d\mathcal{H}^n_{\mathbb{H}^{n+1}}= \frac{c_0}{\epsilon^{n-1}}+\frac{c_2}{\epsilon^{n-3}}+\cdots+\frac{c_{n-2}}{\epsilon}+c_n+o(1).
\end{equation}
If $n$ is odd, the expansion takes the following form:
\begin{equation}\label{37}
    \text{Vol}_\epsilon(\Sigma;r)=\int_{\Sigma\cap\{r\geq\epsilon\}}1\cdot d\mathcal{H}^n_{\mathbb{H}^{n+1}}= \frac{c_0}{\epsilon^{n-1}}+\frac{c_2}{\epsilon^{n-3}}+\cdots+\frac{c_{n-3}}{\epsilon^2}+c_{n-1}\log\frac{1}{\epsilon}+c_{n}+o(1).
\end{equation}
The coefficient $c_n$ is called the \textit{renormalized area} of $\Sigma$ relative to $r$, denoted by $\mathcal{A}(\Sigma;r)$. The authors in \cite{GrahamWitten} further showed that when $n$ is even, $\mathcal{A}(\Sigma;r)$ is conformally invariant. That is, $\mathcal{A}(\Sigma;r)\equiv\mathcal{A}(\Sigma)$ is independent of the choice of boundary defining functions. In \cite{AlexakisMazzeo}, Alexakis-Mazzeo proved that for a $C^2$-asymptotic surface $\Sigma$ in $\mathbb{H}^3$, if $\Sigma$ meets $\partial_\infty\mathbb{H}^3$ orthogonally (which holds for all minimal surfaces), then $\mathcal{A}(\Sigma)$ can be expressed as
\begin{equation*}
    \mathcal{A}(\Sigma)=-2\pi\chi(\Sigma)-\frac{1}{2}\int_{\Sigma}|\hat{A}_\Sigma|^2d\mathcal{H}^2_\mathbb{H},
\end{equation*}
where $\chi(\Sigma)$ is the Euler characteristic and $\hat{A}_\Sigma$ is the trace-free second fundamental form of $\Sigma$. In \cite{BernsteinIsoper}, Bernstein proved an isoperimetric inequality for renormalized area. See also \cite[Theorem 8]{NguyenWeightedMon} for the same result using a different approach. Recently, Tyrrell \cite{TyrrellRenADim4} derived a formula for renormalized area in $\mathbb{H}^5$.

We follow \cite{BWRelativeEntropy} to define a \textit{relative entropy} of two minimal hypersurfaces $\Sigma$ and $\Sigma'$ with $\partial_\infty\Sigma=\partial_\infty\Sigma'$. Let $r$ be any fixed boundary defining function. We consider
\begin{equation}\label{50}
    E^{(r)}_{rel}[\Sigma',\Sigma]:=\lim_{\epsilon\to0^+}\Big(\text{Vol}_\epsilon(\Sigma';r)-\text{Vol}_\epsilon(\Sigma;r)\Big).
\end{equation}
We will show that $E^{(r)}_{rel}[\Sigma',\Sigma]$ is independent of the choice of boundary defining functions, and that the relative entropy $E^{(r)}_{rel}[\Sigma',\Sigma]$ is finite and is equal to the difference of $\mathcal{A}(\Sigma;r)$ and $\mathcal{A}(\Sigma';r)$. It is then worth noting that, although in odd dimensions, $\mathcal{A}(\Sigma;r)$ is in general dependent on the choice of boundary defining functions, the difference, $\mathcal{A}(\Sigma';r)-\mathcal{A}(\Sigma;r)$, is conformally invariant.

The first result is the following:

\begin{thm}\label{35}
Assume $\alpha\in(0,1)$. Let $M$ be a $C^{n,\alpha}$ submanifold in $\partial_\infty\mathbb{H}^{n+1}$ and $r$ be a boundary defining function for $\mathbb{H}^{n+1}$. Let $\Sigma$ and $\Sigma'$ be two minimal hypersurfaces in $\mathbb{H}^{n+1}$ that are both $C^{n,\alpha}$-asymptotic to $M$ $($see $\mathrm{Definition}$ $\ref{66}$ for the precise definition$)$. Then $E^{(r)}_{rel}[\Sigma',\Sigma]$ is a finite number and is independent of the choice of boundary defining function $r$. That is, $E^{(r)}_{rel}[\Sigma',\Sigma]\equiv E_{rel}[\Sigma',\Sigma]$. Moreover,
\begin{equation*}
    E_{rel}[\Sigma',\Sigma]=\mathcal{A}(\Sigma';r)-\mathcal{A}(\Sigma;r).
\end{equation*}
\end{thm}

In Theorem \ref{35}, $C^{n,\alpha}$-asymptotic regularity is to guarantee the existence of the expansions (\ref{36}) and (\ref{37}). However, when $n=2$, the authors in \cite{AlexakisMazzeo} used the Gauss-Codazzi equations and the Gauss-Bonnet formula to prove that for all $C^2$-asymptotic minimal surfaces, the renormalized area is well-defined. So, we conjecture that $C^n$-asymptotic regularity is sufficient to guarantee the existence of the renormalized area.

More generally, we shall prove that the relative entropy $E^{(r)}_{rel}[\Sigma',\Sigma]$ is well-defined for $\Sigma'$ in a ``thin" neighborhood of $\Sigma$. This will be treated in Section 3. Here, by a ``thin'' neighborhood, we roughly mean that the volume of the domain bounded by $\Sigma$ and $\Sigma'$ is finite near the ideal boundary. See Definition \ref{63} for the precise statement.

It is also useful to study a weighted analog of the relative entropy for hypersurfaces in a thin neighborhood of $\Sigma$. This allows us to develop a variational theory for $E_{rel}$. In particular, a mountain pass theorem for $E_{rel}$ is proved in \cite{YaoMountainPass}. We will discuss the weighted relative entropy in Section 4.

Finally, we prove a monotonicity formula for the relative entropy of a mean curvature flow.

\begin{thm}
Let $\{\Sigma_t\}_{t\in[0,T)}$ be a mean curvature flow in $\mathbb{H}^{n+1}$, so that each $\Sigma_t$ is trapped between $\Gamma'_-$ and $\Gamma'_+$. Assume that $\Sigma_0$ is $C^2$-asymptotic to a closed submanifold $M\subset\partial_\infty\mathbb{H}^{n+1}$. Then, there exists an $E_0=E_0(\Sigma_0,\Gamma'_-,\Gamma'_+)>0$, so that for any $0\leq t\leq s<T$, we have\label{mono}
\begin{equation*}
    E_{rel}[\Sigma_t,\Sigma]\geq E_{rel}[\Sigma_s,\Sigma]\geq-E_0.
\end{equation*}
\end{thm}

\subsection*{Acknowledgement}The author would like to thank his advisor, Jacob Bernstein, who made many useful comments on various stages of the work and suggested the computation that appears in the appendix. The author also thanks Letian Chen for stimulating discussions.

\section{Some Preliminaries}
Let $M$ be an $(n-1)$-dimensional embedded (smooth) submanifold in $\partial_\infty\mathbb{H}^{n+1}$. If we fix any point $p_0\in\mathbb{H}^{n+1}$, then there is an isometry $i:\ \mathbb{H}^{n+1}\to\mathbb{R}^{n+1}_+$, so that $i(p_0)=(\textbf{0},1)$ and $i^*g_P=g_{\mathbb{H}^{n+1}}$. Since all such isometries of $\mathbb{H}^{n+1}$ correspond to the M\"obius transformations on $\partial_\infty\mathbb{H}^{n+1}=\mathbb{S}^{n+1}$, we can assume that $M\subset\mathbb{R}^n\times\{0\}$ after a suitable transformation.

\begin{defn}\label{66}
A complete embedded hypersurface $\Sigma\subset\mathbb{H}^{n+1}$ is called $C^{k,\alpha}$-$asymptotic$ to $M$ if $\Sigma\cup M$ (denoted by $\overbar{\Sigma}$) is a $C^{k,\alpha}$ hypersurface in $\overbar{\mathbb{R}_+^{n+1}}$. $M$ is also denoted by $\partial_\infty \Sigma$.
\end{defn}

For convenience, all the unspecified notations below are understood in $\mathbb{R}^{n+1}_+$ with metric $g_P$. Notations with a bar on the top are understood in $\overbar{\mathbb{R}^{n+1}_+}\cup\{\infty\}=\overbar{\mathbb{H}^{n+1}}$ with metric $r^2g_P$. For example, for a hypersurface $\Sigma$, $g_\Sigma$ (resp. $\overbar{g_\Sigma}$) is the Riemannian metric induced from $g_P$ (resp. $r^2g_P$), and $\text{div}_\Sigma$ (resp. $\overbar{\text{div}}_\Sigma$) is referred to the divergence taken with respect to $g_\Sigma$ (resp. $\overbar{g_\Sigma}$). To avoid misunderstanding, for a set $S\subset\mathbb{R}^{n+1}_+$, $\overbar{S}$ is referred to the topological closure in $\overbar{\mathbb{R}^{n+1}_+}\cup\{\infty\}=\overbar{\mathbb{H}^{n+1}}$, and the closure taken in $\mathbb{R}^{n+1}_+$ with metric $g_P$ is denoted by $\text{cl}(S)$. For any two vectors $\textbf{X}$ and $\textbf{Y}$, we think of them as vectors in Euclidean space. So, the addition and substraction make sense. The vector norm in Euclidean space will be denoted by $|\cdot|$ and that in hyperbolic space will be denoted by $|\cdot|_\mathbb{H}$. The inner product $\left\langle\textbf{X},\textbf{Y}\right\rangle$ is taken with respect to $g_P$, and $\textbf{X}\cdot\textbf{Y}$ is the inner product in Euclidean space.

\begin{defn}\label{63}
Given a hypersurface $\Sigma$ which is $C^2$-asymptotic to $M\subset\mathbb{R}^n\times\{0\}\subset\partial_\infty\mathbb{H}^{n+1}$ and meets $M$ orthogonally, a bounded open set $\Omega\subset\mathbb{R}^{n+1}_+$ is called a \textit{thin neighborhood of $\Sigma$}, if there are constants $C, \epsilon>0$, such that $\Omega\cap\{y=y_0\}\subset\overbar{\mathcal{N}}_{C_1y_0^{n+1}}(\Sigma)$ for $0<y_0<\epsilon$. Here, $\overbar{\mathcal{N}}_\delta(\Sigma)$ is the $\delta$-tubular neighborhood of $\Sigma$ considered in Euclidean topology.
\end{defn}

\iffalse
If $M\subset\mathbb{R}^{n}\times\{0\}$ is an ($n-1$)-dimensional $C^3$ closed submanifold and $\Sigma$ is a minimal hypersurface in $\mathbb{H}^{n+1}$ that is $C^3$-asymptotic to $M$, then Lemma \ref{30} implies that any minimal hypersurface $\Sigma'$ with $\partial_\infty\Sigma'=\partial_\infty\Sigma=M$ is contained in a thin neighborhood of $\Sigma$.
\fi

Now, let us fix a minimal hypersurface $\Sigma$ with $C^2$-asymptotic boundary $M$. Let $\Omega$ be a thin neighborhood of $\Sigma$. For simplicity, We choose hypersurfaces $\Gamma'_\pm$ and open sets $U_{\Gamma'_\pm}\subset\mathbb{R}^{n+1}_+$ satisfying $\partial U_{\Gamma'_\pm}=\Gamma'_\pm$ and $U_{\Gamma'_+}\backslash \text{cl}(U_{\Gamma'_-})=\Omega$. In particular, we can choose $\Gamma'_\pm$ to be $C^2$-asymptotic and orthogonal to $M$. Then, $\Omega$ is \textit{quasi-convex} in the sense that for any two points $p_1$, $p_2$ in $\Omega$, there is a path $\gamma$ connecting them, so that the inequality
\begin{equation*}
    \text{Length}(\gamma)\leq C|\textbf{x}(p_1)-\textbf{x}(p_2)|
\end{equation*}
holds for some constant $C$ independent of the choice of points. Here, the length is measured with respect to Euclidean metric, and $\textbf{x}(p_i)$ is the position vector. We fix this choice of $\Gamma'_\pm$ throughout the article. With $\Gamma'_-$ and $\Gamma'_+$ illustrated above, we define
\begin{equation*}
    \mathcal{C}(\Gamma'_-,\Gamma'_+):=\{U:\, U\text{ is a Caccioppoli set}\footnote{See \cite[Section 1]{GiustiBook} for the precise definition.}\text{ and }U_{\Gamma'_-}\subset U\subset U_{\Gamma'_+}\}.
\end{equation*}

We then discuss the space of admissible weights.

\begin{defn}\label{40}
A function $\psi=\psi(p,\textbf{v})$ over $\text{cl}(\Omega)\times\mathbb{S}^n$ is said to be admissible, if
\begin{itemize}
    \item $\psi$ and $\nabla_{\mathbb{S}^n}\psi$ are locally Lipschitz functions over $\text{cl}(\Omega)\times\mathbb{S}^n$;
    \item $\psi(p,\textbf{v})=\psi(p,-\textbf{v})$;
    \item $\|\psi\|_{\mathfrak{X}}<\infty$,
\end{itemize}
where $\|\psi\|_{\mathfrak{X}}:=\|\psi\|_\infty+\|\nabla_{\mathbb{S}^n}\psi\|_\infty+\|\nabla_{\mathbb{S}^n}\nabla_{\mathbb{S}^n}\psi\|_\infty+\|y\nabla_{\mathbb{R}^{n+1}_+}\psi\|_\infty+\|y\nabla_{\mathbb{R}^{n+1}_+}\nabla_{\mathbb{S}^n}\psi\|_\infty$. The space of all admissible weights is denoted by $\mathfrak{X}$.
\end{defn}

Since $\Omega$ is quasi-convex, the sup-norm of the gradient is equivalent to the Lipschitz constant defined by difference quotients. Since $\psi(p,\textbf{v})=\psi(p,-\textbf{v})$, it can be viewed naturally as a function over the Grassman $n$-plane bundle of $\Omega$. It can also be checked that the space $\mathfrak{X}$ is a Banach space and an algebra.

The main tool used in the discussion of the relative entropy is the divergence theorem applied to an appropriate vector field. Since $\Sigma$ is $C^2$-asymptotic to $M$, we choose a small $\delta>0$, so that the normal projection $\Pi:\ \overbar{\mathcal{N}}_\delta(\overbar{\Sigma})\rightarrow\overbar{\Sigma}$ is $C^1$. We fix a small $\epsilon$, so that $\text{cl}(\Omega)\cap\{y<\epsilon\}$ is contained in $\overbar{\mathcal{N}}_\delta(\Sigma)$. Then,
\begin{equation}\label{39}
    \textbf{X}(q)=\textbf{n}_\Sigma\circ \Pi(q)=y\big(\Pi(q)\big)\overbar{\textbf{n}}_\Sigma\circ\Pi(q)
\end{equation}
is a vector field that is well-defined in $\text{cl}(\Omega)\cap\{y<\epsilon\}$. Let $\overbar{\textbf{X}}=\frac{1}{y}\textbf{X}$. This implies $\overbar{\text{div}}\overbar{\textbf{X}}$ is bounded in $\Omega\cap\{y<\epsilon\}$. Let $X^{(n+1)}=\textbf{X}\cdot\overbar{\textbf{e}}_{n+1}=\overbar{\textbf{e}}_{n+1}\cdot(\textbf{n}_\Sigma\circ \Pi)$ be the ($n+1$)-th component of $\textbf{X}$. Then, the assumption that $\Sigma$ meets $\partial_\infty\mathbb{H}^{n+1}$ orthogonally gives
\begin{align}\label{46}
    |X^{(n+1)}|=y\big(\Pi(q)\big)|\overbar{\textbf{n}}^{(n+1)}_\Sigma\circ\Pi(q)|\leq Cy^2\big(\Pi(q)\big)\leq C'y^2(q).
\end{align}
By a direct computation, we have
\begin{equation}\label{41}
    |\text{div}\textbf{X}|=|\overbar{\text{div}}(y\overbar{\textbf{X}})-\frac{n+1}{y}X^{(n+1)}|\leq y|\overbar{\text{div}}\overbar{\textbf{X}}|+\frac{C}{y}|X^{(n+1)}|\leq Cy.
\end{equation}
Since $\Omega$ is a thin neighborhood,
\begin{equation}\label{42}
    \mathcal{H}^n_{\mathbb{H}^{n+1}}(\Omega\cap\{y=y_1\})=\frac{1}{y_1^n}\mathcal{H}^n(\Omega\cap\{y=y_1\})\leq \frac{1}{y_1^n}\cdot Cy_1^{n+1}=Cy_1.
\end{equation}

We shall also consider a family of cut-off functions. Let $\phi_{t_1,\delta}(t)$ be a smooth cut-off function that is identically $1$ when $t\geq t_1$, and vanishes for $t\leq t_1-\delta$. We can further require that $|\phi'_{t_1,\delta}|\leq \frac{C}{\delta}$, where $C$ is independent of $\delta$ and $t_1$. Let $\phi_{t_1,t_2,\delta}:=\phi_{t_1,\delta}-\phi_{t_2+\delta,\delta}$, for any $t_1<t_2$ with $0<\frac{t_1}{2}<t_1-\delta$ and $t_2+\delta<2t_2<\epsilon$. Fix a boundary defining function $r$. Define $\phi^{(r)}_{t_1,\delta}=\phi_{t_1,\delta}\circ r$ and $\phi^{(r)}_{t_1,t_2,\delta}=\phi_{t_1,t_2,\delta}\circ r$. It is readily checked that $\frac{1}{C}\leq |\overbar{\nabla}r|\leq C$ near $\partial_\infty\mathbb{H}^{n+1}$ for some constant $C>0$. Since $\Sigma$ meets $M$ orthogonally and $\overbar{\nabla}r$ is perpendicular to $\partial_\infty\mathbb{H}^{n+1}$ on $M$, we have $\frac{1}{2C}\leq |\overbar{\nabla}_{\overbar{\Sigma}}r|\leq 2C$ near $M$.

By choosing $\epsilon$ sufficiently small, we can assume that in $\Omega\cap\{r\leq\epsilon\}$, $\frac{y}{C}\leq r(\textbf{x},y)\leq Cy$. Hence, (\ref{46}) and (\ref{41}) imply that $|\text{div}\textbf{X}|\leq Ct_1$ and $|\textbf{X}^{(n+1)}|\leq Ct_1^2$ on $\Omega\cap\{r=t_1\}$ with $t_1\leq \epsilon$. Since $\overbar{\nabla}r$ is perpendicular to $\partial_\infty\mathbb{H}^{n+1}$, we also have $\mathcal{H}^n_{\mathbb{H}}(\Omega\cap\{r=t_1\})\leq Ct_1$ on $\Omega\cap\{r=t_1\}$.

Let $U_\Gamma\in\mathcal{C}(\Gamma'_-,\Gamma'_+)$ with $\Gamma=\partial^*U_\Gamma$ being the reduced boundary of $U_\Gamma$. Consider
\begin{equation*}
    E_{rel}[\Gamma,\Sigma;\phi^{(r)}_{t_1,\delta}]:=\int_\Gamma \phi^{(r)}_{t_1,\delta} d\mathcal{H}^n_{\mathbb{H}^{n+1}}-\int_{\Sigma} \phi^{(r)}_{t_1,\delta} d\mathcal{H}^n_{\mathbb{H}^{n+1}}.
\end{equation*}
Letting $\delta\to0^+$, it follows from the dominated convergence theorem that
\begin{equation*}
    E_{rel}[\Gamma,\Sigma;\chi_{\{r\geq t_1\}}]=\lim_{\delta\to 0^+}E_{rel}[\Gamma,\Sigma;\phi^{(r)}_{t_1,\delta}]=\mathcal{H}^n_{\mathbb{H}^{n+1}}(\Gamma\cap\{r\geq t_1\})- \mathcal{H}^n_{\mathbb{H}^{n+1}}(\Sigma\cap\{r\geq t_1\}).
\end{equation*}

\section{The Existence of Relative Entropy\label{32}}
In this section, we prove Theorem \ref{35}. As a first step, we prove that the relative entropy is well-defined for all the reduced boundaries of Caccioppoli sets in $\mathcal{C}(\Gamma'_-,\Gamma'_+)$.

\begin{thm}
For any $U\in\mathcal{C}(\Gamma'_-,\Gamma'_+)$ and any boundary defining function $r$,
\begin{equation*}
E^{(r)}_{rel}[\partial^*U,\Sigma]=\lim_{\epsilon\to0^+}\big(\mathcal{H}^n_\mathbb{H}(\partial^*U\cap\{r\geq\epsilon\})-\mathcal{H}^n_\mathbb{H}(\Sigma\cap\{r\geq\epsilon\})\big)
\end{equation*}
exists $(\in(-\infty,\infty])$ and $E^{(r)}_{rel}[\partial^*U,\Sigma]\equiv E_{rel}[\partial^*U,\Sigma]$ is independent of $r$. Moreover, $E_{rel}[\partial^*U,\Sigma]\geq-E_0$ for $E_0=E_0(\Sigma,\Omega)$. \label{relative entropy}
\end{thm}

\begin{proof}
Since $\overbar{\Sigma}$ is an embedded hypersurface in $\overbar{\mathbb{R}^{n+1}_+}$, it is orientable. Thus, we can choose an open set $U_\Sigma\subset\mathbb{R}^{n+1}_+$ satisfying $\Sigma=\text{cl}(U_\Sigma)\backslash U_\Sigma$ and $U_{\Gamma'_-}\subset U_\Sigma\subset U_{\Gamma'_+}$.

Fix $t_1$, $t_2$ and $\delta$ with $0<\frac{t_1}{2}<t_1-\delta<t_1<t_2<t_2+\delta<2t_2<\epsilon$. Combining the computation in Section 2 and applying the divergence theorem to the vector field $\textbf{X}$, we get
\begin{align}
    &\int_{\partial^*U}\phi^{(r)}_{t_1,t_2,\delta}d\mathcal{H}^n_{\mathbb{H}}-\int_{\Sigma}\phi^{(r)}_{t_1,t_2,\delta}d\mathcal{H}^n_{\mathbb{H}}\geq \int_{\partial^*U}\left\langle\phi^{(r)}_{t_1,t_2,\delta}\textbf{X},\textbf{n}_{\partial^*U}\right\rangle d\mathcal{H}^n_{\mathbb{H}}-\int_{\Sigma}\left\langle\phi^{(r)}_{t_1,t_2,\delta}\textbf{X},\textbf{n}_{\Sigma}\right\rangle d\mathcal{H}^n_{\mathbb{H}}\label{51}\\
    &=\int_{U\backslash U_{\Sigma}}\text{div}\left(\phi^{(r)}_{t_1,t_2,\delta}\textbf{X}\right)d\mathcal{H}^{n+1}_{\mathbb{H}}-\int_{U_{\Sigma}\backslash U}\text{div}\left(\phi^{(r)}_{t_1,t_2,\delta}\textbf{X}\right)\mathcal{H}^{n+1}_{\mathbb{H}}\\
    &\geq -2\int_{\Omega}\left(\phi^{(r)}_{t_1,t_2,\delta}|\text{div}\textbf{X}|+|\left\langle \nabla\phi^{(r)}_{t_1,t_2,\delta}, \textbf{X}\right\rangle|\right)d\mathcal{H}^{n+1}_{\mathbb{H}}\\
    &\geq -C\int_{\Omega\cap\{t_1-\delta<r<t_2+\delta\}}rd\mathcal{H}^{n+1}_{\mathbb{H}}-C\int_{\Omega\cap\left(\{t_1-\delta<r<t_1\}\cup\{t_2<t<t_2+\delta\}\right)}\frac{1}{\delta}|\overbar{\nabla}r\cdot\textbf{X}|d\mathcal{H}^{n+1}_{\mathbb{H}}\label{57}
\end{align}
Since $\overbar{\nabla}r$ is perpendicular to $\mathbb{R}^n\times\{0\}$, we have $|\overbar{\nabla}r\cdot\textbf{X}|\leq C|\textbf{X}^{(n+1)}|\leq Cr^2$. Using the co-area formula for (\ref{57}), we get
\begin{equation}\label{61}
    \int_{\Omega\cap\{t_1-\delta<r<t_2+\delta\}}rd\mathcal{H}^{n+1}_{\mathbb{H}}=\int_{t_1-\delta}^{t_2+\delta}dt\int_{\Omega\cap\{r=t\}}\frac{r}{y^{n+1}|\overbar{\nabla} r|}d\mathcal{H}^n\leq Ct^2_2,
\end{equation}
and similarly,
\begin{equation}\label{62}
    \int_{\Omega\cap\left(\{t_1-\delta<r<t_1\}\cup\{t_2<t<t_2+\delta\}\right)}\frac{1}{\delta}|\overbar{\nabla}r\cdot\textbf{X}|d\mathcal{H}^{n+1}_{\mathbb{H}}\leq Ct^2_2.
\end{equation}
Hence, (\ref{57})$\geq -Ct_2$. After rewriting (\ref{51}), we see that
\begin{equation*}
    E_{rel}[\partial^*U,\Sigma;\phi^{(r)}_{t_1,\delta}]\geq E_{rel}[\partial^*U,\Sigma;\phi^{(r)}_{t_2+\delta,\delta}]-Ct_2.
\end{equation*}
Since $U$ has locally finite perimeter, $\partial^*U$ has finite $\mathcal{H}^n$ measure on each compact set in $\mathbb{R}^{n+1}_+$. Letting $\delta\rightarrow0^+$, the dominated convergence theorem gives
\begin{equation}\label{67}
    E_{rel}[\partial^*U,\Sigma;\chi_{\{r\geq t_1\}}]\geq E_{rel}[\partial^*U,\Sigma;\chi_{\{r\geq t_2\}}]-Ct_2,
\end{equation}
and this implies $\liminf_{t_1\rightarrow 0^+}E_{rel}[\partial^*U,\Sigma;\chi_{\{r\geq t_1\}}]\geq \limsup_{t_2\rightarrow 0^+}E_{rel}[\partial^*U,\Sigma;\chi_{\{r\geq t_2\}}]$. In other words, $E^{(r)}_{rel}[\partial^*U,\Sigma]$ is well-defined.

Fix $t_2<\epsilon$ and let $t_1\to0^+$. We get
\begin{equation*}
    \limsup_{t_1\rightarrow 0^+}E_{rel}[\partial^*U,\Sigma;\chi_{\{r\geq t_1\}}]\geq E_{rel}[\partial^*U,\Sigma;\chi_{\{r\geq t_2\}}]-Ct_2>-\infty.
\end{equation*}
That is, $E^{(r)}_{rel}[\partial^*U,\Sigma]>-\infty$. Notice that
\begin{equation*}
    E_{rel}[\partial^*U,\Sigma;\chi_{\{r\geq t_2\}}]-Ct_2\geq-\mathcal{H}^n_\mathbb{H}(\Sigma\cap\{r\geq t_2\})-Ct_2.
\end{equation*}
We can take $E_0=\mathcal{H}^n_\mathbb{H}(\Sigma\cap\{r\geq t_2\})+Ct_2$.

Finally, given another boundary defining function $\widetilde{r}$, we show $E^{(r)}_{rel}[\partial^*U,\Sigma]=E^{(\widetilde{r})}_{rel}[\partial^*U,\Sigma]$. We have shown that both $E^{(r)}_{rel}[\partial^*U,\Sigma]$ and $E^{(\widetilde{r})}_{rel}[\partial^*U,\Sigma]$ are well-defined. We pick a sequence $\{t_i\}_{i=1}^\infty$ converging to $0$ and appropriately choose another sequence $\{s_i\}_{i=1}^\infty$ satisfying $s_i\to0$ and $\Omega\cap\{r\geq t_i\}\subset\Omega\cap\{\widetilde{r}\geq s_i\}$. For $i$ sufficiently large, we choose $\delta$ sufficiently small, so that $\psi_i^\delta=\phi^{(\widetilde{r})}_{s_i,\delta}-\phi^{(r)}_{t_i+\delta,\delta}$ is well-defined. Then, using a computation similar to (\ref{51})-(\ref{62}),
\begin{align*}
    &\int_{\partial^*U}\psi_i^\delta d\mathcal{H}^n_{\mathbb{H}}-\int_{\Sigma}\psi_i^\delta d\mathcal{H}^n_{\mathbb{H}}\geq-2\int_{\Omega}\left(\psi_i^\delta|\text{div}\textbf{X}|+|\left\langle \nabla\psi_i^\delta, \textbf{X}\right\rangle|\right)d\mathcal{H}^{n+1}_{\mathbb{H}}\\
    \geq&-C\mathcal{H}^{n+1}_\mathbb{H}(\Omega\cap\{r<t_i+\delta\})-\frac{C}{\delta}\mathcal{H}^{n+1}_\mathbb{H}(\Omega\cap\{t_i<r<t_i+\delta\})-\frac{C}{\delta}\mathcal{H}^{n+1}_\mathbb{H}(\Omega\cap\{s_i-\delta<\widetilde{r}<s_i\})\\
    \geq&-C(t_i+s_i).
\end{align*}
This implies $E_{rel}[\partial^*U,\Sigma;\phi^{(\widetilde{r})}_{s_i,\delta}]\geq E_{rel}[\partial^*U,\Sigma;\phi^{(r)}_{t_i+\delta,\delta}]-C(t_i+s_i)$. Letting $\delta\to0$,
\begin{equation*}
    E_{rel}[\partial^*U,\Sigma;\chi_{\{\widetilde{r}\geq s_i\}}]\geq E_{rel}[\partial^*U,\Sigma;\chi_{\{r\geq t_i\}}]-C(t_i+s_i).
\end{equation*}
Letting $i\to\infty$, we set $E^{(\widetilde{r})}_{rel}[\partial^*U,\Sigma]\geq E^{(r)}_{rel}[\partial^*U,\Sigma]$. Interchanging $r$ and $\widetilde{r}$ gives the reverse inequality. Since $r$ and $\widetilde{r}$ are arbitrary, we have $E^{(r)}_{rel}[\partial^*U,\Sigma]\equiv E_{rel}[\partial^*U,\Sigma]$ for any boundary defining function $r$.
\end{proof}

\begin{rem}\label{47}
We see from the proof of Theorem \ref{relative entropy} that $E_{rel}[\partial^*U,\Sigma]$ is still well-defined, even if we replace the condition $\Omega\cap\{y=y_0\}\subset\overbar{\mathcal{N}}_{C_1y_0^{n+1}}(\Sigma)$ by $\Omega\cap\{y=y_0\}\subset\overbar{\mathcal{N}}_{C_1y_0^{n}}(\Sigma)$. However, the thin neighborhood allows us to define the weighted relative entropy for a broader class of weights.
\end{rem}

Let $M$ be an ($n-1$)-dimensional $C^m$ closed submanifold in $\mathbb{R}^{n}\times\{0\}$ with $1\leq m\leq n$. After a suitable transformation, we may assume that $\textbf{0}\in M$ and that the normal vector of $M$ at $\textbf{0}$ in $\mathbb{R}^n\times\{0\}$ is equal to $\overbar{\textbf{e}_n}=(0,\cdots,0,1,0)$. Let $\Sigma$ be a minimal hypersurface with $C^m$-asymptotic boundary $M$. Since $\Sigma$ meets $\mathbb{R}^n\times\{0\}$ orthogonally, we see that for a small $\delta>0$, there is a function $u=u(\textbf{x}',y)$ in $C^{m}\Big((-\delta,\delta)^{n-1}\times[0,\delta)\Big)$, where $\textbf{x}'$ is the first $(n-1)$ coordinates and $y$ is the $(n+1)$-th coordinate, so that
\begin{align}\label{44}
    \Sigma\cap \Big((-\delta,\delta)^{n}\times[0,\delta)\Big)=\{(\textbf{x}',x_n,y):\ x_n=u(\textbf{x}',y),\ (\textbf{x}',y)\in (-\delta,\delta)^{n-1}\times[0,\delta)\}.
\end{align}
The orthogonality condition implies $\partial_y u(\textbf{x}',0)=\textbf{0}$. Since $\Sigma$ is minimal in $\mathbb{H}^{n+1}$, $u$ satisfies the equation
\begin{align}\label{43}
    \partial_{yy}u+\sum_{i=1}^{n-1}\partial_{ii}u-\sum_{j,k\in\{1,\cdots,n-1,y\}}\frac{\partial_{j}u\,\partial_{k}u\,\partial_{jk}u}{1+(\partial_yu)^2+\sum_{i=1}^{n-1}(\partial_i u)^2}-\frac{n}{y}\partial_y u=0.
\end{align}

\begin{lem}\label{45}
Assume $M$ is $C^{m}$ with $1\leq m\leq n$ and $\Sigma$ is a minimal hypersurface that is $C^{m}$-asymptotic to $M$. If we write $\Sigma$ locally as a graph as in $(\ref{44})$, then $\partial^{(i)}_yu(\mathbf{x}',0)$ is determined by $M$ for $0\leq i\leq k$.
\end{lem}

\begin{proof}
Assume $m\geq2$, otherwise there is nothing to prove. Fix $1\leq k\leq m-1$. Assume that $\partial^{(i)}_y u(\textbf{x}',0)$ is determined by the boundary data, $u(\textbf{x}',0)$, for $1\leq i\leq k$, which is true for $k=1$. We prove by induction that the $(k+1)$-th $y$-derivative at $y=0$ is also determined by the boundary data. Multiply both sides of (\ref{43}) by $y$ and take the $(k-1)$-th $y$-derivative on both sides. We get
\begin{align*}
    y\partial^{(k+1)}_y u(\textbf{x}',y)+(k-1-n)\partial^{(k)}_y u(\textbf{x}'y)+\frac{y\partial_{y}u\,\partial_{y}u\,\partial^{(k+1)}_{y}u}{1+(\partial_yu)^2+\sum_{i=1}^{n-1}(\partial_i u)^2}+f+g\partial^{(k)}_y u=0,
\end{align*}
where $f$ and $g$ are rational functions of $\partial^{\alpha_1}_{x_1}\cdots\partial^{\alpha_{n-1}}_{x_{n-1}}\partial^{(i)}_y u$ ($1\leq i\leq k-1$) with $g\big|_{y=0}=0$. As a result,
\begin{align*}
    (k-n)\partial^{(k+1)}_y u(\textbf{x}',0)+\partial_yf+\partial_yg\,\partial^{(k)}_y u=0,\text{ for }\textbf{x}'\in(-\delta,\delta)^{n-1}.
\end{align*}
Hence, by induction, $\partial^{(i)}_y u(\textbf{x}',0)$ is determined by the boundary data when $1\leq i\leq m$.
\end{proof}

Let $\Gamma_1$ and $\Gamma_2$ be two minimal hypersurfaces that are both $C^{3}$-asymptotic to $M\subset\mathbb{R}^n\times\{0\}$. In the next lemma, we show that $\Gamma_2$ is in a thin neighborhood of $\Gamma_1$.

\begin{lem}\label{30}
Suppose $\Gamma_1$, $\Gamma_2$ are two minimal hypersurfaces in $\mathbb{H}^{n+1}$ that are both $C^3$-asymptotic to $M\subset\mathbb{R}^{n}\times\{0\}\subset\partial_\infty\mathbb{H}^{n+1}$. Then, there exist an $\epsilon_1>0$ and a function $h$, so that
\begin{equation}\label{78}
    \Gamma_2\cap\{y\leq\frac{\epsilon_1}{2}\}\subset\big\{p+h(p)\overbar{\mathbf{n}}_{\Gamma_1}:p\in\Gamma_1\cap\{y\leq\epsilon_1\}\big\}\subset\Gamma_2\cap\{y\leq2\epsilon_1\},
\end{equation}
where $h$ satisfies the estimate $|h(p)|\leq Cy^{n+1}(p)$ for $C>0$ and $p\in\Gamma_1\cap\{y\leq\epsilon_1\}$.
\end{lem}

\begin{proof}
Since both $\Gamma_1$ and $\Gamma_2$ meet $M$ orthogonally and the second fundamental form, $|\overbar{A_{\Gamma_i}}|$, is bounded near the ideal boundary, it follows that there are an $\epsilon_1>0$ and a function $h$ satisfying (\ref{78}). Moreover, $h$ is $C^2$ up to $M$. By Lemma \ref{45}, for a small $\epsilon>0$, when writing $\Gamma_i$ as graphs over $M\times[0,\epsilon)$ near $\{y=0\}$, the $y$-derivatives are determined by $M$ up to the second order. So,
\begin{equation*}
    h\big|_M=0,\ \overbar{\nabla}h\big|_M=0\text{ and }\overbar{\nabla}^2h\big|_M=0.
\end{equation*}
Set $\Gamma_{1,\epsilon_1}:=\Gamma_1\cap\{y\leq\epsilon_1\}$. By a direct computation, $f=\frac{h}{y}$ satisfies the equation
\begin{equation*}
    \Delta f+(|A_{\Gamma_1}|^2_\mathbb{H}-n)f=a(p,h,\overbar{\nabla}h,\overbar{\nabla}^2h)f+y\left\langle\textbf{b}(p,h,\overbar{\nabla}h,\overbar{\nabla}^2h),\nabla f\right\rangle,
\end{equation*}
where $a\in C^0(\Gamma_{1,\epsilon_1}\times\mathbb{R}\times \mathbb{R}^{n+1}\times\mathbb{R}^{(n+1)^2};\mathbb{R})$, $b\in C^0(\Gamma_{1,\epsilon_1}\times\mathbb{R}\times \mathbb{R}^{n+1}\times\mathbb{R}^{(n+1)^2};\mathbb{R}^{n+1})$ satisfy
\begin{equation*}
    |a|+|\textbf{b}|\leq C(|h|+|\overbar{\nabla}h|+|\overbar{\nabla}^2h|).
\end{equation*}
So, $a\big|_M=0$ and $\textbf{b}\big|_M=\textbf{0}$.

Since $|A_{\Gamma_1}|^2_\mathbb{H}=\mathcal{O}(y^2)$, we can shrink $\epsilon_1$ if necessary, so that $|A_{\Gamma_1}|^2_\mathbb{H}-n-a\leq0$ for $y(p)\leq\epsilon_1$. That is, the operator $\mathcal{L}=\Delta-y\left\langle\textbf{b},\nabla\right\rangle+(|A_{\Gamma_1}|^2_\mathbb{H}-n-a)$ satisfies the maximum principle in $\Gamma_{1,\epsilon_1}$. Let $g(p)=y^{n}(p)-Ky^{n+1}(p)$($K$ to be determined). Then, near the ideal boundary, 
\begin{align}
   \mathcal{L}(y^{n}(p))&=\Delta y^{n}(p)-y\left\langle\textbf{b},\nabla y^{n}(p)\right\rangle+(|A_{\Gamma_1}|^2_\mathbb{H}-n-a)y^{n}(p)\\
   &=y^2\partial_{yy}(y^{n})+(2-n)y\partial_y(y^{n})-ny^{n}+\mathcal{O}(y^{n+1})=\mathcal{O}(y^{n+1}).\nonumber
\end{align}
Similarly, we have $\mathcal{L}(y^{n+1})=(n+2)y^{n+1}+\mathcal{O}(y^{n+2})$. So, by further shrinking $\epsilon_1$ and choosing $K$ appropriately, we have $\mathcal{L}g=\mathcal{L}(y^n-Ky^{n+1})\leq0$. Let $\lambda=\sup_{\{y(p)=\epsilon_1\}}f^+\geq0$ and $g_1=\frac{\lambda}{\epsilon_1^n-K\epsilon_1^{n+1}}g$. Then,
\begin{equation*}
    \mathcal{L}(f-g_1)\geq0,\ f\big|_{y=0}=0=g_1\big|_{y=0}\text{ and }f\big|_{y=\epsilon_1}\leq g_1\big|_{y=\epsilon_1}.
\end{equation*}
By the maximum principle, $f\leq g_1$ in $\Gamma_1\cap\{y\leq\epsilon_1\}$. This proves the upper bound.

The lower bound follows in a similar manner if we let $\lambda'=\sup_{\{y=\epsilon_1\}}f^-\geq0$ and consider $\mathcal{L}(f+\frac{\lambda'}{\epsilon_1^n-K\epsilon_1^{n+1}}g)\leq0$.
\end{proof}

\begin{proof}[Proof of Theorem \ref{35}]
By Lemma \ref{30}, $\Sigma'$ is in a thin neighborhood of $\Sigma$ when $n\geq3$. If $n=2$, by Remark \ref{47}, $E_{rel}[\Sigma',\Sigma]$ is still well-defined. It follows from Theorem \ref{relative entropy} that $E^{(r)}_{rel}[\Sigma',\Sigma]=E_{rel}[\Sigma',\Sigma]$ for any boundary defining function $r$. Moreover, we have $E_{rel}[\Sigma',\Sigma]>-\infty$. Similarly, $-E_{rel}[\Sigma',\Sigma]=E_{rel}[\Sigma,\Sigma']>-\infty$. This implies $E_{rel}[\Sigma',\Sigma]$ is a finite value.

Fix any boundary defining function $r$ and consider $\mathcal{H}^n_\mathbb{H}(\Sigma\cap\{r\geq\epsilon\})$. We choose an $\epsilon_1>0$, so that $\frac{1}{C}\leq|\overbar{\nabla}^{\overbar{\Sigma}} r|\leq C$ for $r\leq\epsilon_1$. By the coarea formula,
\begin{align}
    \mathcal{H}^n_\mathbb{H}(\Sigma\cap\{r\geq\epsilon\})&=\mathcal{H}^n_\mathbb{H}(\Sigma\cap\{r\geq\epsilon_1\})+\mathcal{H}^n_\mathbb{H}(\Sigma\cap\{\epsilon\leq r<\epsilon_1\})\label{53}\\
    &=\mathcal{H}^n_\mathbb{H}(\Sigma\cap\{r\geq\epsilon_1\})+\int_\epsilon^{\epsilon_1}\int_{\Sigma\cap\{r=s\}}\frac{1}{|\nabla^\Sigma r|}d\mathcal{H}^{n-1}_\mathbb{H}ds,\nonumber
\end{align}
where $|\nabla^\Sigma r|=r|\overbar{\nabla}_r^{\overbar{\Sigma}}r|$. Here, $\overbar{\nabla}_r$ is the gradient taken in $(\overbar{\mathbb{H}}^{n+1},r^2g_P)$. So,
\begin{equation}\label{48}
    \int_{\Sigma\cap\{r=s\}}\frac{1}{|\nabla^\Sigma r|}d\mathcal{H}^{n-1}_\mathbb{H}=\frac{1}{s^{n}}\int_{\Sigma\cap\{r=s\}}\frac{1}{|\overbar{\nabla}_r^{\overbar{\Sigma}}r|}d\mathcal{H}_r^{n-1},
\end{equation}
where $\mathcal{H}_r^{n-1}$ is the ($n-1$)-dimensional Hausdorff measure in ($\overbar{\mathbb{H}}^{n+1}$,$r^2g_P$). Since $\Sigma$ is $C^{n,\alpha}$-asymptotic to $M$, the quantity $\int_{\Sigma\cap\{r=s\}}\frac{1}{|\overbar{\nabla}_r^{\overbar{\Sigma}}r|}d\mathcal{H}_r^{n-1}$ is a $C^{n-1,\alpha}$ function in $s\in[0,\epsilon_1)$. Thus, taking the Taylor's expansion at $s=0$,
\begin{equation}\label{49}
    \int_{\Sigma\cap\{r=s\}}\frac{1}{|\overbar{\nabla}_r^{\overbar{\Sigma}}r|}d\mathcal{H}_r^{n-1}=a_0+a_1s+\cdots+a_{n-1}s^{n-1}+\mathcal{O}(s^{n-1+\alpha}).
\end{equation}
Combining (\ref{48}) and (\ref{49}), we see that
\begin{align}\label{52}
    \int_{\Sigma\cap\{r=s\}}\frac{1}{|\nabla^\Sigma r|}d\mathcal{H}^{n-1}_\mathbb{H}=\frac{1}{s^n}\big(a_0+a_1s+\cdots+a_{n-1}s^{n-1}+\mathcal{O}(s^{n-1+\alpha})\big).
\end{align}
Plugging (\ref{52}) into (\ref{53}), we get an expansion for $\mathcal{H}^n_\mathbb{H}(\Sigma\cap\{r\geq\epsilon\})$:
\begin{equation*}
    \mathcal{H}^n_\mathbb{H}(\Sigma\cap\{r\geq\epsilon\})=\frac{b_0}{\epsilon^{n-1}}+\cdots+\frac{b_{n-2}}{\epsilon}+b_{n-1}\log\left(\frac{1}{\epsilon}\right)+b_n+o(1).
\end{equation*}
This proves the existence of the renormalized area $\mathcal{A}(\Sigma;r)$ (equal to $b_n$) when $M$ is $C^{n,\alpha}$. Similarly, we have
\begin{equation*}
    \mathcal{H}^n_\mathbb{H}(\Sigma'\cap\{r\geq\epsilon\})=\frac{b'_0}{\epsilon^{n-1}}+\cdots+\frac{b'_{n-2}}{\epsilon}+b'_{n-1}\log\left(\frac{1}{\epsilon}\right)+b'_n+o(1).
\end{equation*}
By definition, $E_{rel}[\Sigma',\Sigma]=\lim_{\epsilon\to0}\Big(\mathcal{H}^n_\mathbb{H}(\Sigma'\cap\{r\geq\epsilon\})-\mathcal{H}^n_\mathbb{H}(\Sigma\cap\{r\geq\epsilon\})\Big)$. Since $E_{rel}[\Sigma',\Sigma]$ is finite, we see that $b_i=b'_i$ for $0\leq i\leq n-1$ and that $E_{rel}[\Sigma',\Sigma]=b'_n-b_n=\mathcal{A}(\Sigma';r)-\mathcal{A}(\Sigma;r)$.
\end{proof}

\begin{rem}
In Appendix \ref{64}, we will give an explicit formula of the relative entropy for geodesics in hyperbolic plane and verify that it is indeed independent of the choice of boundary defining functions.
\end{rem}

\section{Weighted Relative Entropy\label{33}}

In this section, we study a weighted version of the relative entropy. Since the thin neighborhood $\Omega$ is bounded in $\mathbb{R}^{n+1}_+$, we can pick a boundary defining function $r_1$ satisfying $r_1(\textbf{x},y)=y$ in a compact subset of $\overbar{\mathbb{R}^{n+1}_+}$ containing $\Omega$. Thus, 
\begin{equation*}
    E_{rel}[\Sigma',\Sigma]=E^{(r_1)}_{rel}[\Sigma',\Sigma]=\lim_{\epsilon\to0^+}\big(\mathcal{H}^n_\mathbb{H}(\Sigma'\cap\{y\geq\epsilon\})-\mathcal{H}^n_\mathbb{H}(\Sigma\cap\{y\geq\epsilon\})\big).
\end{equation*}
For simplicity, we will always use this specific boundary defining function when computing the relative entropy. For any function $\psi$, let
\begin{align}\label{56}
    E_{rel}[\partial^*U,\Sigma;\psi]:=\lim_{\epsilon\to0^+}E_{rel}[\partial^*U,\Sigma;\psi\chi_{\{y\geq\epsilon\}}]=\lim_{\epsilon\to0^+}\Big(\int_{\partial^*U\cap\{y\geq\epsilon\}}\psi d\mathcal{H}^n_{\mathbb{H}}-\int_{\Sigma\cap\{y\geq\epsilon\}}\psi d\mathcal{H}^n_{\mathbb{H}}\Big)
\end{align}
whenever it exists. Observe that (\ref{56}) is always well-defined for any compactly supported continuous function whose support is away from $\mathbb{R}^n\times\{0\}$.

The main theorem of this section is the following:

\begin{thm}
Assume that $U\in\mathcal{C}(\Gamma'_-,\Gamma'_+)$ satisfies $E_{rel}[\partial^*U,\Sigma]<\infty$. Then, for any $\psi(p,\mathbf{v})\in\mathfrak{X}$, the limit $(\ref{56})$ exists. If we denote it by $E_{rel}[\partial^*U,\Sigma;\psi]$, then
\begin{equation}
    |E_{rel}[\partial^*U,\Sigma;\psi]|\leq C(|E_{rel}[\partial^*U,\Sigma]|+1)\|\psi\|_{\mathfrak{X}},\label{10}
\end{equation}
where the constant $C$ only depends on $\Sigma$ and $\Omega$.\label{weighted entropy} Moreover, for $0<\frac{y_1}{2}<y_1-\delta<y_1<y_2<y_2+\delta<2y_2<\epsilon$,
\begin{equation}\label{73}
     |E_{rel}[\Gamma,\Sigma;\phi_{y_1,y_2,\delta}\psi]|\leq C(|E_{rel}[\Gamma,\Sigma;\phi_{y_1,y_2,\delta}]|+y_2)\|\psi\|_{\mathfrak{X}}.
\end{equation}
where $\epsilon$ is made so small that the vector field $\mathbf{X}$ defined in $(\ref{39})$ exists.
\end{thm}

The arguments in \cite[Section 4]{BWRelativeEntropy} can be carried out in our setting. In \cite{BWRelativeEntropy}, the authors first proved that if the weights are quadratic forms, then the weighted relative entropy is well-defined. Then, for a general weight, they showed that by subtracting off an appropriate quadratic form from it, the error term has a reasonably good estimate. We follow the same scheme.

We first prove a weak version of Theorem \ref{weighted entropy} that will be used later on.

\begin{lem}\label{lip lem}
Let $U_\Gamma\in\mathcal{C}(\Gamma'_-,\Gamma'_+)$ with $\Gamma=\partial^*U_\Gamma$. If $\psi$ is a locally Lipschitz function over $\mathbb{R}^{n+1}_+$ with $\|\psi\|_\infty+\|y\overbar{\nabla}\psi\|_\infty\leq 1$, then for $0<\frac{y_1}{2}<y_1-\delta<y_1<y_2<y_2+\delta<2y_2<\epsilon$,
\begin{equation}
    \big|E_{rel}[\Gamma,\Sigma;\psi\phi_{y_1,y_2,\delta}]\big|\leq\big| E_{rel}[\Gamma,\Sigma;\phi_{y_1,y_2,\delta}]\big|+Cy_2,\label{vari for lip}
\end{equation}
where $\epsilon$ is the same as in Theorem \ref{weighted entropy}.
\end{lem}

\begin{proof}
We first assume $\psi$ is nonnegative. Then, similar to the proof of Theorem \ref{relative entropy}, we have
\begin{align*}
    &\int_{\Gamma}\psi\phi_{y_1,y_2,\delta}d\mathcal{H}^n_{\mathbb{H}^{n+1}}-\int_{\Sigma}\psi\phi_{y_1,y_2,\delta}d\mathcal{H}^n_{\mathbb{H}^{n+1}}\\
    \geq& -2\int_{\Omega}|\text{div}\left(\psi\phi_{y_1,y_2,\delta}\textbf{X}\right)|d\mathcal{H}^{n+1}_{\mathbb{H}^{n+1}}\\
    \geq& -2\int_{\Omega}\psi\phi_{y_1,y_2,\delta}|\text{div}\textbf{X}|+\psi|\left\langle\nabla\phi_{y_1,y_2,\delta},\textbf{X}\right\rangle|+\phi_{y_1,y_2,\delta}|\left\langle\nabla\psi,\textbf{X}\right\rangle|d\mathcal{H}^{n+1}_{\mathbb{H}^{n+1}}.
\end{align*}
Direct computations show that $\psi\phi_{y_1,y_2,\delta}|\text{div}\textbf{X}|\leq C\phi_{y_1,y_2,\delta}$ and
\begin{align*}
    |\left\langle \nabla\psi,\textbf{X}\right\rangle|=|\overbar{\nabla}\psi\cdot\textbf{X}|=|\big(y\overbar{\nabla}\psi\big)\cdot\big(\frac{\textbf{X}}{y}\big)|\leq C,
\end{align*}
Thus, applying the divergence theorem to the vector field $\textbf{X}$ in a similar way as in the proof of Theorem \ref{relative entropy},
\begin{align*}
     &\int_{\Gamma}\psi\phi_{y_1,y_2,\delta}d\mathcal{H}^n_{\mathbb{H}^{n+1}}-\int_{\Sigma}\psi\phi_{y_1,y_2,\delta}d\mathcal{H}^n_{\mathbb{H}^{n+1}}\\
     \geq& -C\mathcal{H}^n_{\mathbb{H}^{n+1}}\Big(\Omega\cap\{y_1-\delta\leq y\leq y_2+\delta\}\Big)-C\int_{\Omega\cap\big(\{y_1-\delta\leq y\leq y_1\}\cup\{y_2\leq y\leq y_2+\delta\}\big)}\frac{1}{\delta}d\mathcal{H}^n_{\mathbb{H}^{n+1}}\\
     \geq& -Cy_2.
\end{align*}
This gives $E_{rel}[\Gamma,\Sigma;\psi\phi_{y_1,y_2,\delta}]\geq-Cy_2$. If we replace $\psi$ by $1-\psi$ and notice that $1-\psi\geq0$, then
\begin{align*}
    E_{rel}[\Gamma,\Sigma;\psi\phi_{y_1,y_2,\delta}]=E_{rel}[\Gamma,\Sigma;\phi_{y_1,y_2,\delta}]-E_{rel}[\Gamma,\Sigma;(1-\psi)\phi_{y_1,y_2,\delta}]\leq E_{rel}[\Gamma,\Sigma;\phi_{y_1,y_2,\delta}]+Cy_2.
\end{align*}

For a general $\psi\in\mathfrak{X}$ with $\|\psi\|_\mathfrak{X}\leq 1$, we have $\frac{1+\psi}{2}\geq 0$ and $\|\frac{1+\psi}{2}\|_\mathfrak{X}\leq1$. So, (\ref{vari for lip}) follows if we apply the preceding argument to $\frac{1+\psi}{2}$.
\end{proof}

It is then immediate from the next proposition that Theorem \ref{weighted entropy} holds for weights $\psi(p,\textbf{v})\in\mathfrak{X}$ that is independent of $\textbf{v}$.

\begin{prop}\label{4}
Let $U_\Gamma\in\mathcal{C}(\Gamma'_-,\Gamma'_+)$ with $\Gamma=\partial^*U_\Gamma$. Assume that $E_{rel}[\Gamma,\Sigma]$ is finite and that $\psi$ is a locally Lipschitz function over $\mathbb{R}^{n+1}_+$ with $\|\psi\|_w:=\|\psi\|_\infty+\|y\overbar{\nabla}\psi\|_\infty<\infty$. Then,
\begin{align}\label{68}
    \big|E_{rel}[\Gamma,\Sigma,\psi]\big|\leq C\Big(\big|E_{rel}[\Gamma,\Sigma]\big|+1\Big)\|\psi\|_w,
\end{align}
where $C$ is independent of $\Gamma=\partial^*U$ and $\psi$.
\end{prop}

\begin{proof}
For simplicity, we assume that $\|\psi\|_w=1$ if $\psi$ is nonzero. Otherwise, the conclusion follows trivially. In Lemma \ref{lip lem}, by letting $\delta\to0^+$, we get
\begin{align}
    \Big|E_{rel}[\Gamma,\Sigma;\psi\chi_{\{y\geq y_1\}}]-E_{rel}[\Gamma,\Sigma;\psi\chi_{\{y\geq y_2\}}]\Big|\leq Cy_2+\Big|E_{rel}[\Gamma,\Sigma;\chi_{\{y\geq y_1\}}]-E_{rel}[\Gamma,\Sigma;\chi_{\{y\geq y_2\}}]\Big|.\label{1}
\end{align}
Since $E_{rel}[\Gamma,\Sigma]=\lim_{\epsilon\to0^+}E_{rel}[\Gamma,\Sigma;\chi_{\{y\geq \epsilon\}}]$ exists and is a finite number, the right hand side of (\ref{1}) can be made arbitrarily small. So, Cauchy's criterion implies that
\begin{align*}
    E_{rel}[\Gamma,\Sigma;\psi]=\lim_{\epsilon\to0^+}E_{rel}[\Gamma,\Sigma;\psi\chi_{\{y\geq \epsilon\}}]
\end{align*}
is well-defined and is finite valued. Taking $r=r_1$ in (\ref{67}), we see that
\begin{equation*}
    E_{rel}[\Gamma,\Sigma;\chi_{\{y\geq y_1\}}]\geq E_{rel}[\Gamma,\Sigma;\chi_{\{y\geq y_2\}}]-Cy_2
\end{equation*}
holds for $y_1<y_2<\epsilon$. We choose $y_1<\epsilon$ so small that $\big|E_{rel}[\Gamma,\Sigma;\chi_{\{y\geq y_1\}}]\big|\leq 2\big|E_{rel}[\Gamma,\Sigma]\big|$. Note that $y_1$ may depend on $\Gamma$. However, for any $y_2<\epsilon$,
\begin{equation*}
    E_{rel}[\Gamma,\Sigma;\chi_{\{y\geq y_2\}}]\leq Cy_2+E_{rel}[\Gamma,\Sigma;\chi_{\{y\geq y_1\}}]\leq Cy_2+2\big|E_{rel}[\Gamma,\Sigma]\big|.
\end{equation*}
Letting $y_1\to0$ in (\ref{1}), we get
\begin{align}
    \big|E_{rel}[\Gamma,\Sigma;\psi]\big|&\leq Cy_2+\Big|E_{rel}[\Gamma,\Sigma]-E_{rel}[\Gamma,\Sigma;\chi_{\{y\geq y_2\}}]\Big|+|E_{rel}[\Gamma,\Sigma;\psi\chi_{\{y\geq y_2\}}]|\label{2}\\
    & \leq Cy_2+3|E_{rel}[\Gamma,\Sigma]|+|E_{rel}[\Gamma,\Sigma;\psi\chi_{\{y\geq y_2\}}]|.\nonumber
\end{align}
Notice that
\begin{align}
    |E_{rel}[\Gamma,\Sigma;\psi\chi_{\{y\geq y_2\}}]|=&\Big|\int_{\Gamma\cap\{y\geq y_2\}}\psi d\mathcal{H}^n_{\mathbb{H}}-\int_{\Sigma\cap\{y\geq y_2\}}\psi d\mathcal{H}^n_{\mathbb{H}}\Big|\\
    \leq&\max\Big\{\big|\int_{\Gamma\cap\{y\geq y_2\}}\psi d\mathcal{H}^n_{\mathbb{H}}\big|,\big|\int_{\Sigma\cap\{y\geq y_2\}}\psi d\mathcal{H}^n_{\mathbb{H}}\big|\Big\}\\
    \leq&\mathcal{H}^n_{\mathbb{H}}(\Sigma\cap\{y\geq y_2\})+|E_{rel}[\Gamma,\Sigma;\chi_{\{y\geq y_2\}}|\\
    \leq&2\mathcal{H}^n_{\mathbb{H}}(\Sigma\cap\{y\geq y_2\})+2|E_{rel}[\Gamma,\Sigma]|.\label{3}
\end{align}
Since $y_2<\epsilon$ is fixed and $\mathcal{H}^n_{\mathbb{H}}(\Sigma\cap\{y\geq y_2\})<\infty$, combining (\ref{2})-(\ref{3}) gives the desired upper bound. Since $y_2$ is independent of $\Gamma$, so is the constant $C$ in (\ref{68}).
\end{proof}

Now, we follow \cite{BWRelativeEntropy} to establish the estimate for quadratic forms.

\begin{lem}\label{69}
Let $\mathbf{Y}_1$ and $\mathbf{Y}_2$ be two vector fields in $\mathbb{R}^{n+1}_+$ with
\begin{equation*}
    \|\mathbf{Y}_i\|_\infty+\|y\overbar{\nabla}\mathbf{Y}_i\|_\infty\leq1.
\end{equation*}
Then, for $0<\frac{y_1}{2}<y_1-\delta<y_1<y_2<y_2+\delta<2y_2<\epsilon$, we have
\begin{equation}
    |E_{rel}[\Gamma,\Sigma;\phi_{y_1,y_2,\delta}(\mathbf{Y}_1\cdot\mathbf{v})(\mathbf{Y}_2\cdot\mathbf{v})]|\leq C(y_2+|E_{rel}[\Gamma,\Sigma;\phi_{y_1,y_2,\delta}]|).\label{esti for quad}
\end{equation}
\end{lem}

\begin{proof}
Let $\textbf{Y}$ be a locally Lipschitz vector field on $\mathbb{R}^{n+1}_+$ with $\|\textbf{Y}\|_\infty+\|y\overbar{\nabla}\textbf{Y}\|_\infty\leq1$. We first show that Lemma \ref{69} holds for $\textbf{Y}_1=\textbf{Y}_2=\textbf{Y}$. By a direct computation,
\begin{equation*}
    |(\textbf{Y}\cdot\textbf{X})\textbf{Y}^{(n+1)}|\leq Cy,\ \left|\overbar{\text{div}}\Big((\textbf{Y}\cdot\textbf{X})\textbf{Y}\Big)\right|\leq C.
\end{equation*}
Applying the divergence theorem to the vector field $(\textbf{X}\cdot\textbf{Y})\textbf{Y}$ in a similar way as in the proof of Theorem \ref{relative entropy},
\begin{align*}
    &\int_{\Gamma}\phi_{y_1,y_2,\delta}(\textbf{Y}\cdot\textbf{X})\left\langle\textbf{Y},\textbf{n}_\Gamma\right\rangle d\mathcal{H}^n_{\mathbb{H}^{n+1}}- \int_{\Sigma}\phi_{y_1,y_2,\delta}(\textbf{Y}\cdot\textbf{X})\left\langle\textbf{Y},\textbf{n}_{\Sigma}\right\rangle d\mathcal{H}^n_{\mathbb{H}^{n+1}}\\
    &\geq -2\int_{\Omega}\left|\text{div}\Big(\phi_{y_1,y_2,\delta}(\textbf{X}\cdot\textbf{Y})\textbf{Y}\Big)\right|d\mathcal{H}^{n+1}_{\mathbb{H}^{n+1}}\\
    &\geq -C y_2.
\end{align*}
Using the inequality $(\textbf{Y}\cdot\textbf{X})\left\langle\textbf{Y},\textbf{n}_\Gamma\right\rangle=(\textbf{Y}\cdot\overbar{\textbf{X}})(\textbf{Y}\cdot \overbar{\textbf{n}}_\Gamma)\leq\frac{1}{2}(\textbf{Y}\cdot\overbar{\textbf{X}})^2+\frac{1}{2}(\textbf{Y}\cdot\overbar{\textbf{n}}_\Gamma)^2$ and noticing that $\textbf{Y}\cdot\overbar{\textbf{X}}=\textbf{Y}\cdot\overbar{\textbf{n}}_\Gamma$ over $\Sigma\cap\{y\leq\epsilon\}$, one gets
\begin{align}
    &\int_{\Gamma}\phi_{y_1,y_2,\delta}(\textbf{Y}\cdot\textbf{X})\left\langle\textbf{Y},\textbf{n}_\Gamma\right\rangle d\mathcal{H}^n_{\mathbb{H}^{n+1}}- \int_{\Sigma}\phi_{y_1,y_2,\delta}(\textbf{Y}\cdot\textbf{X})\left\langle\textbf{Y},\textbf{n}_{\Sigma}\right\rangle d\mathcal{H}^n_{\mathbb{H}^{n+1}}\\
    \leq& \int_{\Gamma}\phi_{y_1,y_2,\delta}\left(\frac{1}{2}(\textbf{Y}\cdot\overbar{\textbf{X}})^2+\frac{1}{2}(\textbf{Y}\cdot\overbar{\textbf{n}}_\Gamma)^2\right) d\mathcal{H}^n_{\mathbb{H}^{n+1}}- \int_{\Sigma}\phi_{y_1,y_2,\delta}\left(\frac{1}{2}(\textbf{Y}\cdot\overbar{\textbf{X}})^2+\frac{1}{2}(\textbf{Y}\cdot\overbar{\textbf{n}}_{\Sigma})^2\right) d\mathcal{H}^n_{\mathbb{H}^{n+1}}\\
    =&E_{rel}[\Gamma,\Sigma;\frac{1}{2}\phi_{y_1,y_2,\delta}(\textbf{Y}\cdot\overbar{\textbf{X}})^2]+E_{rel}[\Gamma,\Sigma;\frac{1}{2}\phi_{y_1,y_2,\delta}(\textbf{Y}\cdot\textbf{v})^2].\label{70}
\end{align}
Notice that $\frac{1}{2}(\textbf{Y}\cdot\overbar{\textbf{X}})^2$ satisfies
\begin{equation*}
    \|(\textbf{Y}\cdot\overbar{\textbf{X}})^2\|_\infty+\|y\overbar{\nabla}\big(\textbf{Y}\cdot\overbar{\textbf{X}}\big)^2\|_\infty\leq C.
\end{equation*}
So, applying Lemma \ref{lip lem} to the first term in (\ref{70}), it is bounded by $Cy_2+C|E_{rel}[\Gamma,\Sigma;\phi_{y_1,y_2,\delta}]|$. Hence,
\begin{equation}
    E_{rel}[\Gamma,\Sigma;\phi_{y_1,y_2,\delta}(\textbf{Y}\cdot\textbf{v})^2]\geq -Cy_2-C|E_{rel}[\Gamma,\Sigma;\phi_{y_1,y_2,\delta}]|.\label{5}
\end{equation}
If we are given two vector fields $\textbf{Y}_1$ and $\textbf{Y}_2$ with
\begin{align*}
    \|\textbf{Y}_i\|_\infty+\|y\overbar{\nabla}\textbf{Y}_i\|_\infty\leq1\text{ for }i=1,2,
\end{align*}
then, using (\ref{5}) we have
\begin{align*}
    &E_{rel}[\Gamma,\Sigma;\phi_{y_1,y_2,\delta}\big((\textbf{Y}_1+\textbf{Y}_2)\cdot\textbf{v}\big)^2]\\
    =&2\sum_{i=1}^2 E_{rel}[\Gamma,\Sigma;\phi_{y_1,y_2,\delta}(\textbf{Y}_1\cdot\textbf{v})^2]-4E_{rel}[\Gamma,\Sigma;\phi_{y_1,y_2,\delta}\big(\frac{\textbf{Y}_1-\textbf{Y}_2}{2}\cdot\textbf{v}\big)^2]\\
    \leq& 2\sum_{i=1}^2 E_{rel}[\Gamma,\Sigma;\phi_{y_1,y_2,\delta}(\textbf{Y}_1\cdot\textbf{v})^2]+4\Big(Cy_2+C|E_{rel}[\Gamma,\Sigma;\phi_{y_1,y_2,\delta}]|\Big).
\end{align*}
Now, assume $\textbf{Y}_i$($i=1,\cdots,k$) are vector fields with $\|\textbf{Y}_i\|_\infty+\|y\overbar{\nabla}\textbf{Y}_i\|_\infty\leq1$. By induction, it is immediate that
\begin{align}
     E_{rel}[\Gamma,\Sigma;\phi_{y_1,y_2,\delta}\big((\sum_{i=1}^k\textbf{Y}_i)\cdot\textbf{v}\big)^2]\leq2^k\sum_{i=1}^k E_{rel}[\Gamma,\Sigma;\phi_{y_1,y_2,\delta}\big(\textbf{Y}_i\cdot\textbf{v}\big)^2]+C_k(|E_{rel}[\Gamma,\Sigma;\phi_{y_1,y_2,\delta}]|+y_2).\label{6}
\end{align}
If we write $\textbf{Y}=\sum_{i=1}^{n+1}Y_i\overbar{\textbf{e}}_i$, then
\begin{equation*}
    \big((Y_i\overbar{\textbf{e}}_i)\cdot \textbf{v}\big)^2=Y_i^2-\sum_{m\neq i}\big((Y_i\overbar{\textbf{e}}_m)\cdot\textbf{v}\big)^2.
\end{equation*}
If we take $\psi=Y_i^2$ in Lemma \ref{lip lem} and use (\ref{5}) with $(\textbf{Y}\cdot\textbf{v})^2$ replaced by $((Y_i\textbf{e}_m)\cdot \textbf{v})^2$, then
\begin{align*}
    E_{rel}[\Gamma,\Sigma;\phi_{y_1,y_2,\delta}\big((Y_i\overbar{\textbf{e}}_i)\cdot \textbf{v}\big)^2]\leq C(y_2+|E_{rel}[\Gamma,\Sigma;\phi_{y_1,y_2,\delta}]|).
\end{align*}
Together with (\ref{6}), we get
\begin{align*}
    E_{rel}[\Gamma,\Sigma;\phi_{y_1,y_2,\delta}\big(\textbf{Y}\cdot \textbf{v}\big)^2]=E_{rel}[\Gamma,\Sigma;\phi_{y_1,y_2,\delta}\big((\sum_{i=1}^{n+1}Y_i\overbar{\textbf{e}}_i)\cdot \textbf{v}\big)^2]\leq C(y_2+|E_{rel}[\Gamma,\Sigma;\phi_{y_1,y_2,\delta}]|).
\end{align*}
Combining this with (\ref{5}), we see that
\begin{align*}
    \big|E_{rel}[\Gamma,\Sigma;\phi_{y_1,y_2,\delta}\big(\textbf{Y}\cdot \textbf{v}\big)^2]\big|\leq C(y_2+\big|E_{rel}[\Gamma,\Sigma;\phi_{y_1,y_2,\delta}]\big|).
\end{align*}
Then, (\ref{esti for quad}) follows from the following observation:
\begin{equation*}
    (\textbf{Y}_1\cdot\textbf{v})(\textbf{Y}_2\cdot\textbf{v})=\left(\frac{\textbf{Y}_1+\textbf{Y}_2}{2}\cdot\textbf{v}\right)^2-\left(\frac{\textbf{Y}_1-\textbf{Y}_2}{2}\cdot\textbf{v}\right)^2.
\end{equation*}
\end{proof}

Now, we discuss the quadratic approximation for a general weight $\psi\in\mathfrak{X}$ and establish the estimate for the error term. Fix a small $\epsilon>0$. Let $\overbar{\psi}(p,\textbf{v}):=\psi(p,\textbf{v})-(1-\phi_{\frac{\epsilon}{2},\epsilon})(\textbf{Y}_\psi(p)\cdot\textbf{v})(\overbar{\textbf{X}}(p)\cdot\textbf{v})$, where $\textbf{Y}_\psi(p):=(\nabla_{\mathbb{S}^n}\psi)(p,\overbar{\textbf{X}}(p))$. Notice that if $\epsilon$ is sufficiently small, then $\overbar{\textbf{X}}$ is well-defined. Fix a point $p$ in $\{y<\frac{\epsilon}{2}\}$. Notice that $\nabla_{\mathbb{S}^n}\overbar{\psi}(p,\textbf{v})\big|_{\textbf{v}=\overbar{\textbf{X}}(p)}=0$. Let $\gamma(t)$ ($t\in[0,\ell]$) be a geodesic on $\mathbb{S}^n$ connecting $\overbar{\textbf{X}}(p)$ and $\textbf{v}$, where $\ell$ is the length of $\gamma$. If we let $\varphi(t)=\overbar{\psi}(p,\gamma(t))$, then
\begin{align*}
    \varphi'(0)=\nabla_{\mathbb{S}^n}\overbar{
    \psi}\big|_{\textbf{v}=\overbar{\textbf{X}}(p)}\cdot\gamma'(0)=0.
\end{align*}
Hence,
\begin{align}
    |\overbar{\psi}(p,\textbf{v})-\overbar{\psi}(p,\overbar{\textbf{X}}(p))|=&|\varphi(\ell)-\varphi(0)|=|\int_0^\ell \varphi'(t)dt|=|\int_0^\ell\int_0^t\varphi''(s)dsdt|\\
    =&|\int_0^\ell\int_0^t\nabla_{\mathbb{S}^n}\nabla_{\mathbb{S}^n}\psi(\gamma'(s),\gamma'(s)) dsdt|\\
    \leq&\|\nabla_{\mathbb{S}^n}\nabla_{\mathbb{S}^n}\psi\|_\infty\ell^2=\|\nabla_{\mathbb{S}^n}\nabla_{\mathbb{S}^n}\psi\|_\infty d^2_{\mathbb{S}^n}(\textbf{v},\overbar{\textbf{X}}(p)).\label{58}
\end{align}
Since $\psi$ is an even function in $\textbf{v}$, so is $\overbar{\psi}$. Thus,
\begin{align}\label{59}
    |\overbar{\psi}(p,\textbf{v})-\overbar{\psi}(p,\overbar{\textbf{X}}(p))|=|\overbar{\psi}(p,-\textbf{v})-\overbar{\psi}(p,\overbar{\textbf{X}}(p))|\leq\|\nabla_{\mathbb{S}^n}\nabla_{\mathbb{S}^n}\psi\|_\infty d^2_{\mathbb{S}^n}(-\textbf{v},\overbar{\textbf{X}}(p)).
\end{align}
Combining (\ref{58}) and (\ref{59}), we get
\begin{align}
    |\overbar{\psi}(p,\textbf{v})-\overbar{\psi}(p,\overbar{\textbf{X}}(p))|\leq& \|\nabla_{\mathbb{S}^n}\nabla_{\mathbb{S}^n}\psi\|_\infty\min\Big\{d^2_{\mathbb{S}^n}(\textbf{v},\overbar{\textbf{X}}(p)),d^2_{\mathbb{S}^n}(-\textbf{v},\overbar{\textbf{X}}(p))\Big\}\nonumber\\
    \leq& C_n\|\nabla_{\mathbb{S}^n}\nabla_{\mathbb{S}^n}\psi\|_\infty\big(1-(\textbf{v}\cdot\overbar{\textbf{X}})^2\big).\label{7}
\end{align}
By a direct computation, we also have
\begin{align}
    &\|\overbar{\psi}(p,\overbar{\textbf{X}}(p))\|_w\leq C(\|\psi\|_\infty+\|y\nabla_{\mathbb{R}^{n+1}_+}\psi\|_\infty+\|\nabla_{\mathbb{S}^n}\psi\|_\infty+\|y\nabla_{\mathbb{R}^{n+1}_+}\nabla_{\mathbb{S}^n}\psi\|_{\infty}),\label{8}\\
    &\|\textbf{Y}_\psi\|_\infty+\|y\nabla_{\mathbb{R}^{n+1}_+}\textbf{Y}_\psi\|_\infty\leq C(\|\nabla_{\mathbb{S}^n}\psi\|_\infty+\|y\nabla_{\mathbb{R}^{n+1}_+}\nabla_{\mathbb{S}^n}\psi\|_\infty+\|\nabla_{\mathbb{S}^n}\nabla_{\mathbb{S}^n}\psi||_\infty).\label{9}
\end{align}

Now, we are ready to prove Theorem \ref{weighted entropy}.
\begin{proof}[Proof of Theorem \ref{weighted entropy}]
Fix $\psi\in\mathfrak{X}$. Without loss of generality, we can assume that $\|\psi\|_{\mathfrak{X}}\leq 1$. Let $\epsilon$ be the same number that appears in the definition of $\overbar{\psi}$. For positive constants $y_1$, $y_2$ and $\delta$ with
\begin{align*}
    0<\frac{y_1}{2}<y_1-\delta<y_1<y_2<y_2+\delta<2y_2<\frac{\epsilon}{2},
\end{align*}
we prove that
\begin{align}
    |E_{rel}[\Gamma,\Sigma;\phi_{y_1,y_2,\delta}\psi]|\leq C(|E_{rel}[\Gamma,\Sigma;\phi_{y_1,y_2,\delta}]|+y_2)\|\psi\|_{\mathfrak{X}}.\label{11}
\end{align}
This gives the existence of $E_{rel}[\Gamma,\Sigma,\psi]$. Observe that $E_{rel}[\Gamma,\Sigma;\phi_{y_1,y_2,\delta}\psi]$ can be written as 
\begin{align*}
    E_{rel}[\Gamma,\Sigma;\phi_{y_1,y_2,\delta}\psi]=&E_{rel}[\Gamma,\Sigma;\phi_{y_1,y_2,\delta}(\textbf{Y}_\psi(p)\cdot\textbf{v})(\overbar{\textbf{X}}(p)\cdot\textbf{v})]+E_{rel}[\Gamma,\Sigma;\phi_{y_1,y_2,\delta}\overbar{\psi}(p,\overbar{\textbf{X}}(p))]\\
    &+E_{rel}[\Gamma,\Sigma;\phi_{y_1,y_2,\delta}\left(\overbar{\psi}(p,\textbf{v})-\overbar{\psi}(p,\overbar{\textbf{X}}(p))\right)].
\end{align*}
By (\ref{esti for quad}) and (\ref{9}),we have
\begin{align*}
    |E_{rel}[\Gamma,\Sigma;\phi_{y_1,y_2,\delta}(\textbf{Y}_\psi(p)\cdot\textbf{v})(\overbar{\textbf{X}}(p)\cdot\textbf{v})]|\leq C(y_2+|E_{rel}[\Gamma,\Sigma;\phi_{y_1,y_2,\delta}]|)\|\psi\|_{\mathfrak{X}}.
\end{align*}
Proposition \ref{4} and (\ref{8}) gives
\begin{align*}
    |E_{rel}[\Gamma,\Sigma;\phi_{y_1,y_2,\delta}\overbar{\psi}(p,\overbar{\textbf{X}}(p))]|\leq C(y_2+|E_{rel}[\Gamma,\Sigma;\phi_{y_1,y_2,\delta}]|)\|\psi\|_{\mathfrak{X}}.
\end{align*}
Finally, by (\ref{7}) and noticing that both $\left(\overbar{\psi}(p,\textbf{v})-\overbar{\psi}(p,\overbar{\textbf{X}}(p))\right)$ and $\left(1-(\textbf{v}\cdot\overbar{\textbf{X}})^2\right)$ vanish on $\Sigma$,
\begin{align*}
    &\big|E_{rel}[\Gamma,\Sigma;\phi_{y_1,y_2,\delta}\left(\overbar{\psi}(p,\textbf{v})-\overbar{\psi}(p,\overbar{\textbf{X}}(p))\right)]\big|=\left|\int_{\Gamma}\phi_{y_1,y_2,\delta}\left(\overbar{\psi}(p,\textbf{v})-\overbar{\psi}(p,\overbar{\textbf{X}}(p))\right)d\mathcal{H}^n_{\mathbb{H}^{n+1}}\right|\\
    &\leq C\|\psi\|_{\mathfrak{X}}\left|\int_{\Gamma}\phi_{y_1,y_2,\delta}\big(1-(\overbar{\textbf{n}}_\Gamma\cdot\overbar{\textbf{X}})^2\big)d\mathcal{H}^n_{\mathbb{H}^{n+1}}\right|=C\|\psi\|_{\mathfrak{X}}\,\big|E_{rel}[\Gamma,\Sigma;\phi_{y_1,y_2,\delta}\big(1-(\textbf{v}\cdot\overbar{\textbf{X}})^2\big)]\big|\\
    &\leq C\|\psi\|_{\mathfrak{X}}(y_2+|E_{rel}[\Gamma,\Sigma;\phi_{y_1,y_2,\delta}]|).
\end{align*}
This proves (\ref{11}). It follows from Cauchy's criterion that $E_{rel}[\Gamma,\Sigma;\psi]$ is well-defined. We use an argument similar to (\ref{2})-(\ref{3}) in the proof of Proposition \ref{4} to conclude that the inequality (\ref{10}) holds.
\end{proof}

\section{Monotonicity Formula for Relative Entropy\label{34}}

In this section, we prove Theorem \ref{mono}. As a starting point, we show that the flow is $C^{1,1}$-asymptotic to the ideal boundary as long as it exists.

\begin{lem}
Let $\{\Sigma_t\}_{t\in[0,T)}$ be a mean curvature flow in $\mathbb{H}^{n+1}$ trapped between $\Gamma'_-$ and $\Gamma'_+$. Assume $\Sigma_0$ is $C^2$-asymptotic to a closed submanifold $M\subset\partial_\infty\mathbb{H}^{n+1}$. Then, for all $t\in[0,T)$, $\Sigma_t$ is $C^{1,1}$-asymptotic to $M$. Moreover, there is a constant $\epsilon_0>0$, so that
\begin{equation}\label{55}
    \sup_{0\leq t<T,\,p\in\Sigma_t\cap\{y\leq\epsilon_0\}}|\overbar{A_{\Sigma_t}}|<\infty.
\end{equation}
\end{lem}

\begin{proof}
Since $\Sigma_0$ is $C^2$-asymptotic to $M$, we have $|A_{\Sigma_0}|_\mathbb{H}=\mathcal{O}(y)$ as $y\to0$. So, by the pseudo-locality theorem \cite[Theorem 7.5]{ChenPseudolcality}, for any $T_0<T$, there are $\epsilon_0>0$ and $C_0>0$, so that $|A_{\Sigma_t}|_\mathbb{H}(p)\leq C_0$ for all $t\in[0,T_0]$ and $p\in\Sigma_t\cap\{y\leq\epsilon_0\}$. This implies that $|\overbar{A_{\Sigma_t}}|\leq \frac{C}{y}$ for $t\in[0,T_0]$ and $p\in\Sigma_t\cap\{y\leq\epsilon_0\}$.

For $p_0\in\Sigma_t$, consider the ball $B_{\delta y_0}(p_0)$, where $\delta<1$ is a small constant and $y_0=y(p_0)$. We have $|\overbar{A_{\Sigma_t}}|\leq \frac{C}{y_0}$ in $B_{\delta y_0}(p_0)\cap\Sigma_t$. Since $\Omega$ is a thin neighborhood and both $\Gamma'_-$ and $\Gamma'_+$ meet $M$ orthogonally, by shrinking $\epsilon_0$ if necessary, we can assume that $|\overbar{e}_{n+1}\cdot\overbar{n}_{\Sigma_t}|\leq\frac{1}{2}$ and that $\Sigma_t\cap\{y\leq\epsilon_0\}$ is a graph over $M\times(0,\epsilon_0)$. Otherwise, $B_{\delta y_0}(p_0)\cap\Sigma_t$ would intersect the boundary of $\Omega$.

By a suitable transformation, we may assume $\textbf{0}\in M$ and that the unit normal of $M$ at $\textbf{0}$ is equal to $\overbar{e}_{n}=(0,\cdots,0,1,0)$. Thus, for a sufficiently small $\delta>0$, the mean curvature flow $\{\Sigma_t\}_{t\in[0,T_0]}$ in $(-\delta,\delta)^{n}\times(0,\epsilon_0)$ can be described by
\begin{equation}\label{79}
    \Sigma_t\cap\Big((-\delta,\delta)^{n}\times(0,\epsilon_0)\Big)=\{(\textbf{x}',u,y):\,u=u(t,\textbf{x}',y),\,\textbf{x}'\in(-\delta,\delta)^{n-1},\,y\in(0,\epsilon_0)\},
\end{equation}
where $u$ satisfies the equation
\begin{equation}\label{54}
    \partial_tu=y^2\sqrt{1+|Du|^2}\text{div}\left(\frac{Du}{\sqrt{1+|Du|^2}}\right)-ny\partial_yu.
\end{equation}
The equation (\ref{54}) is scaling invariant in the sense that $u_\lambda(t,\textbf{x}',y):=\frac{1}{\lambda}u(t,\lambda\textbf{x}',\lambda y)$ satisfies the same equation. Since $\Sigma_t$ is trapped between $\Gamma'_-$ and $\Gamma'_+$ and $\Sigma$ meets $M$ orthogonally, we have the estimate $|u(t,\textbf{x}',y)|\leq Cy^2$. Let $k\in\mathbb{N}$ and $\lambda_k=\epsilon_0^k$. Then, $u_{\lambda_k}$ solves (\ref{54}) in $[0,T_0]\times(-\frac{\delta}{\lambda_k},\frac{\delta}{\lambda_k})^{n-1}\times[\epsilon_0,1)$. It follows from the gradient estimates for quasi-linear equations that $Du$ is H\"older continuous (see, e.g., \cite[Theorem 6.1.1]{LadyzenskajaLinearQuasil} for a proof). Then, we can apply the standard Schauder theory to $u_{\lambda_k}$ to conclude that, when rescaling back to $u$, $|Du|\leq Cy$ and $|D^2u|\leq C$. Since $T_0$ is arbitrary, this implies each $\Sigma_t$ is $C^{1,1}$-asymptotic to $M$.

Now, fix a specific choice of $T_0<T$ and the corresponding $\epsilon_0$. Let
\begin{equation*}
    S=\big\{s\in(0,T): \Sigma_t\cap\{y\leq\epsilon_0\}\text{ is a graph over }M\times(0,\epsilon_0]\text{ for all }t\in[0,s]\big\}
\end{equation*}
and $T'=\sup S$. Then, $T'\geq T_0$. We claim that $T'=T$. If, on the contrary, $T'<T$, then by definition, $\Sigma_t\cap\{y\leq\epsilon_0\}$ is a graph over $M\times(0,\epsilon_0]$ for $t\in[0,T')$. So, locally we can write $\Sigma_t\cap\Big((-\delta,\delta)^{n}\times(0,\epsilon_0)\Big)$ as a graph of $u$ as in (\ref{79}) for $t\in[0,T)$, where $u$ satisfies the equation in (\ref{54}) in $[0,T)\times\Big((-\delta,\delta)^{n}\times(0,\epsilon_0)$. Then, the estimates on $u$ in the preceding paragraph implies that
\begin{equation*}
    \sup_{t\in[0,T'),p\in\Sigma_t\cap\{y\leq\epsilon_0\}}|A_{\Sigma_t}|<\infty.
\end{equation*}
This implies the graph can be extended to $t=T'$, and hence $T'\in S$. Finally, as $|Du|$ remains bounded for $t\in[0,T']$, we can see that $T'+\epsilon\in S$ for some $\epsilon$. This contradicts our assumption.

So, there is a $\epsilon_0>0$, so that $\Sigma_t\cap\{y\leq\epsilon_0\}$ can be expressed as graphs over $M\times(0,\epsilon_0]$ for all $t<T$. Then, (\ref{55}) follows from the estimates on the equation (\ref{54}).
\end{proof}

\begin{lem}
Let $\widetilde{\Sigma}$ be a hypersurface that is $C^{1,1}$-asymptotic to $M$ and is trapped between $\Gamma'_-$ and $\Gamma'_+$. If $\widetilde{\Sigma}$ satisfies\label{22}
\begin{equation*}
    \sup_{p\in\widetilde{\Sigma}\cap\{y\leq\epsilon_0\}}|\overbar{A}_{\widetilde{\Sigma}}|(p)\leq C_0
\end{equation*}
for some $C_0,\ \epsilon_0>0$, then there exist constants $\epsilon_1=\epsilon_1(C_0,\epsilon_0, \Gamma'_-,\Gamma'_+)$ and $C_1=C_1(C_0,\epsilon_0, \Gamma'_-,\Gamma'_+)$, so that for any $0<\frac{s_1}{2}<s_1-\delta<s_1<s_2<s_2+\delta<2s_2<\epsilon_1$,
\begin{equation*}
    \big|E_{rel}[\widetilde{\Sigma},\Sigma;\phi_{s_1,s_2,\delta}]\big|\leq C_1s_2.
\end{equation*}
\end{lem}

\begin{proof}
Since $\Gamma'_-$ and $\Gamma'_+$ are both orthogonal to $\partial_\infty\mathbb{H}^{n+1}$, so is $\widetilde{\Sigma}$. Together with the upper bound for $|\overbar{A}_{\widetilde{\Sigma}}|$, we have the existence of constant $\epsilon_1$, so $\widetilde{\Sigma}\cap\{y\leq\epsilon_1\}$ can be written as a graph over $\Sigma$. That is, for some $C^{1,1}$ function $f$ defined on $\Sigma\cap\{y\leq2\epsilon_1\}$,
\begin{equation*}
    \widetilde{\Sigma}\cap\{y\leq\epsilon_1\}\subset\big\{p+f(p)\overbar{\textbf{v}}(p):\,\textup{p}\in\Sigma\cap\{y\leq2\epsilon_1\}\big\}\subset\widetilde{\Sigma},
\end{equation*}
where $\overbar{\textbf{v}}(p)$ is a $C^{1,1}$ transverse section\footnote{See \cite[Section 2.4]{BWSpace} for the precise definition.} over $\overbar{\Sigma}$ satisfying $\overbar{\textbf{v}}(p)=\overbar{\textbf{n}}_M$ on $M$. The assumptions on $\widetilde{\Sigma}$ immediately give
\begin{equation*}
    |f(p)|\leq Cy^{n+1}(p),\ |\overbar{\nabla}_{\Sigma}f(p)|\leq Cy(p),\ |\overbar{\nabla}^2_{\Sigma}f(p)|\leq C.
\end{equation*}

For $0\leq s\leq1$, we let $\hat{\Sigma}_s=\big\{\textbf{f}_s(p)=p+sf(p)\overbar{\textbf{v}}(p):\,p\in\Sigma\cap\{y\leq2\epsilon_1\}\big\}$. Then, by the first variation formula,
\begin{align}
    \frac{d}{ds}\int_{\hat{\Sigma}_s}\phi_{s_1,s_2,\delta}d\mathcal{H}^n_{\mathbb{H}}=\int_{\hat{\Sigma}_s}\left\langle\nabla\phi_{s_1,s_2,\delta},\textbf{Y}_s^\perp\right\rangle-\phi_{s_1,s_2,\delta}\left\langle\textbf{Y}_s,\textbf{H}_s\right\rangle d\mathcal{H}^n_{\mathbb{H}},\label{20}
\end{align}
where $\textbf{H}_s$ is the mean curvature of $\hat{\Sigma}_s$ and $\textbf{Y}_s=(f\overbar{n}_{\Sigma})\circ\textbf{f}_s^{-1}$. By shrinking $\epsilon_1$ if needed, we can assume that $y(p)$ is comparable to $y(\textbf{f}_s^{-1}(p))$ for $0\leq s\leq1$ and $p\in\hat{\Sigma}_s$. We also have $\textbf{H}_s=y^2\overbar{\textbf{H}}_s+ny(\overbar{e}_{n+1}\cdot\overbar{n}_{\hat{\Sigma}_s})\overbar{n}_{\hat{\Sigma}_s}=\mathcal{O}(y^2)$, where $\mathcal{O}(y^2)$ is uniform in $s$. Hence,
\begin{align}
    &\left|\left\langle\nabla\phi_{s_1,s_2,\delta},\textbf{Y}_s^\perp\right\rangle\right|\leq \frac{C}{\delta}y^{n+1}(p)(\chi_{\{s_1-\delta\leq y\leq s_1\}}+\chi_{\{s_2\leq y\leq s_2+\delta\}});\label{71}\\ &\left|\phi_{s_1,s_2,\delta}\left\langle\textbf{Y}_s,\textbf{H}_s\right\rangle\right|\leq Cy^{n+1}(p)\chi_{\{y\leq s_2+\delta\}}.\label{72}
\end{align}
Then, applying the coarea formula to the right hand side of (\ref{20}),
\begin{align*}
    (\ref{20})=\int_{s_1-\delta}^{s_2+\delta}\frac{1}{r^n}dr\int_{\hat{\Sigma}_s\cap\{y=r\}}\left(\left\langle\nabla\phi_{s_1,s_2,\delta},\textbf{Y}_s^\perp\right\rangle-\phi_{s_1,s_2,\delta}\left\langle\textbf{Y}_s,\textbf{H}_s\right\rangle\right)\frac{1}{|\overbar{\nabla}_{\hat{\Sigma}_s}y|}d\mathcal{H}^{n-1}.
\end{align*}
Combining this with (\ref{71}) and (\ref{72}), we get (\ref{20}) is bounded by $C_1s_2$. Finally, if we take the integral in (\ref{20}) from $s=0$ to $s=1$, the desired estimates follows.
\end{proof}

Now, let's turn to the proof of the monotonicity formula in Theorem \ref{mono}. In fact, we prove the following theorem, of which Theorem \ref{mono} is a special case.
\begin{thm}
Let $\{\Sigma_t\}_{t\in[0,T)}$ be a mean curvature flow in $\mathbb{H}^{n+1}$. We assume that each $\Sigma_t$ is trapped between $\Gamma'_-$ and $\Gamma'_+$ and that $\Sigma_0$ is $C^2$-asymptotic to $M$. Then, $E_{rel}[\Sigma_t,\Sigma]$ exists and is finite for all $0\leq t<T$. There is a constant $E_0=E_0(\Sigma_0,\Gamma'_-,\Gamma'_+)>0$, so that $E_{rel}[\Sigma_t,\Sigma]\geq -E_0$ for all $t\in[0,T)$. Furthermore, for any $0\leq t\leq s<T$ and non-negative $f=f(x,t)$ defined for $(x,t)\in\mathbb{R}^{n+1}_+\times[0,T)$ with
\begin{equation*}
    \sum_{i=0}^3\|(y\nabla_{\mathbb{R}^{n+1}_+})^if\|_\infty+\sum_{i=0}^1\|(y\nabla_{\mathbb{R}^{n+1}_+})^i\partial_t f\|_\infty<\infty,
\end{equation*}
we have\label{21}
\begin{align}\label{23}
    E_{rel}[\Sigma_t,\Sigma;f]=& E_{rel}[\Sigma_s,\Sigma;f]+\int_t^s dr\int_{\Sigma_r}f\left\langle\mathbf{H},\mathbf{H}\right\rangle d\mathcal{H}^n_\mathbb{H}\\
    &-\int_t^s E_{rel}[\Sigma_r,\Sigma;\partial_t f-\Delta f+\left\langle\nabla_{\mathbf{n}_{\Sigma_r}}\nabla f,\mathbf{n}_{\Sigma_r}\right\rangle]dr.\nonumber
\end{align}
\end{thm}

\begin{proof}
From Lemma \ref{55} and Lemma \ref{22}, we know that there is $\epsilon_1>0$, so that for any $0<s_1\leq s_2<\epsilon_1$,
\begin{equation}
    \big|E_{rel}[\Sigma_t,\Sigma;\chi_{\{y\geq s_1\}}]-E_{rel}[\Sigma_t,\Sigma;\chi_{\{y\geq s_2\}}]\big|\leq C_1s_2.\label{24}
\end{equation}
This implies the existence and finiteness of $E_{rel}[\Sigma_t,\Sigma]$. In (\ref{24}), letting $s_1\to0^+$, we have
\begin{equation*}
    E_{rel}[\Sigma_t,\Sigma]\geq E_{rel}[\Sigma_t,\Sigma;\chi_{\{y\geq s_2\}}]-C_1s_2.
\end{equation*}
On the other hand,
\begin{equation*}
    E_{rel}[\Sigma_t,\Sigma;\chi_{\{y\geq s_2\}}]=\int_{\Sigma_t}\chi_{\{y\geq s_2\}} d\mathcal{H}^n_\mathbb{H}-\int_{\Sigma}\chi_{\{y\geq s_2\}}d\mathcal{H}^n_\mathbb{H}\geq-\mathcal{H}^n_\mathbb{H}(\Sigma\cap\{y\geq s_2\}).
\end{equation*}
Thus, we can take $E_0=\mathcal{H}^n_\mathbb{H}(\Sigma\cap\{y\geq s_2\})+C_1s_2$.

Now, let us take a smooth cut-off function $\eta_\epsilon(y)=\eta_\epsilon(y(p))$ that is $1$ when $y\geq\epsilon$ and vanishes when $y\leq\frac{\epsilon}{2}$. We can further assume that $\big|\epsilon^k\eta^{(k)}_\epsilon\big|\leq C_k$ for any $k$. Then, computing the first variation, we get
\begin{align}\label{60}
    \frac{d}{dt}E_{rel}[\Sigma_t,\Sigma;\eta_\epsilon f]=E_{rel}[\Sigma_t,\Sigma;\partial_t(\eta_\epsilon f)+\left\langle\nabla (\eta_\epsilon f),\textbf{H}\right\rangle]-\int_{\Sigma_t}\eta_\epsilon f\left\langle\textbf{H},\textbf{H}\right\rangle d\mathcal{H}^n_\mathbb{H}.
\end{align}
Here, we have used the minimality of $\Sigma$. Integrating (\ref{60}) from $t$ to $s$, we get
\begin{align*}
    &E_{rel}[\Sigma_s,\Sigma;\eta_\epsilon f]-E_{rel}[\Sigma_t,\Sigma;\eta_\epsilon f]\\
    =&-\int_t^s dr\int_{\Sigma_r}\eta_\epsilon f\left\langle\textbf{H},\textbf{H}\right\rangle d\mathcal{H}^n_\mathbb{H}+\int_t^s\left(E_{rel}[\Sigma_r,\Sigma;\partial_t(\eta_\epsilon f)+\left\langle\nabla (\eta_\epsilon f),\textbf{H}\right\rangle]\right)dr\\
    =&-\int_t^s dr\int_{\Sigma_r}\eta_\epsilon f\left\langle\textbf{H},\textbf{H}\right\rangle d\mathcal{H}^n_\mathbb{H}+\int_t^s E_{rel}[\Sigma_r,\Sigma;\partial_t(\eta_\epsilon f)-\Delta(\eta_\epsilon f)+\left\langle\nabla_{\textbf{n}_{\Sigma_r}}\nabla(\eta_\epsilon f),\textbf{n}_{\Sigma_r}\right\rangle]dr.
\end{align*}
From the assumptions on $f$ and $\eta_\epsilon$, we know that $\partial_t(\eta_\epsilon f)-\Delta(\eta_\epsilon f)+\left\langle\nabla_{\textbf{n}_{\Sigma_r}}\nabla(\eta_\epsilon f),\textbf{n}_{\Sigma_r}\right\rangle$ is uniformly bounded in $\mathfrak{X}$ and pointwisely converges to $\partial_t f-\Delta f+\left\langle\nabla_{\textbf{n}_{\Sigma_r}}\nabla f,\textbf{n}_{\Sigma_r}\right\rangle$. The lemma below shows that the corresponding relative entropy also converges. By the same reasoning, we have
\begin{equation*}
    E_{rel}[\Sigma_s,\Sigma;\eta_\epsilon f]\to E_{rel}[\Sigma_s,\Sigma;f],\ E_{rel}[\Sigma_t,\Sigma;\eta_\epsilon f]\to E_{rel}[\Sigma_t,\Sigma;f]
\end{equation*}
as $\epsilon\to0^+$. The monotone convergence theorem gives
\begin{equation*}
    \int_t^s dr\int_{\Sigma_r}\eta_\epsilon f\left\langle\textbf{H},\textbf{H}\right\rangle d\mathcal{H}^n_\mathbb{H}\to\int_t^s dr\int_{\Sigma_r} f\left\langle\textbf{H},\textbf{H}\right\rangle d\mathcal{H}^n_\mathbb{H}.
\end{equation*}
Putting them together, we get (\ref{23}).
\end{proof}

\begin{lem}
If $f$ and $f_i$ are functions in $\mathfrak{X}$ with uniformly bounded $\mathfrak{X}$-norms and $f_i$ converges pointwise to $f$, then we have $E_{rel}[\partial^*U,\Sigma;f]=\lim_{i\rightarrow\infty}E_{rel}[\partial^*U,\Sigma;f_i]$ for any $U\in\mathcal{C}(\Gamma'_-,\Gamma'_+)$ with $E_{rel}[\partial^*U,\Sigma]<\infty$.
\end{lem}

\begin{proof}
Assume $\|f\|_\mathfrak{X}\leq C$ and $\|f_i\|_\mathfrak{X}\leq C$ for all $i$. For any $y_1, y_2$ small with $0<\frac{y_1}{2}<y_1-\delta<y_1<y_2<y_2+\delta<2y_2<\epsilon$, where $\epsilon$ is the same as in Theorem \ref{weighted entropy},
\begin{align*}
    |E_{rel}[\partial^*U,\Sigma;(f_i-f)\phi_{y_1,y_2,\delta}]|&\leq C(|E_{rel}[\partial^*U,\Sigma;\phi_{y_1,y_2,\delta}]|+y_2)\|f_i-f\|_\mathfrak{X}\\
    &\leq C (|E_{rel}[\partial^*U,\Sigma;\phi_{y_1,y_2,\delta}]|+y_2).
\end{align*}
Letting $\delta\to0^+$ and $y_1\to0^+$,
\begin{align*}
    &|E_{rel}[\partial^*U,\Sigma;(f_i-f)]-E_{rel}[\partial^*U,\Sigma;(f_i-f)\chi_{\{y\geq y_2\}}]|\\
    &\leq C(|E_{rel}[\partial^*U,\Sigma]-E_{rel}[\partial^*U,\Sigma;\chi_{\{y\geq y_2\}}]|+y_2).
\end{align*}
This implies
\begin{align*}
    \big|E_{rel}[\partial^*U,\Sigma;(f_i-f)]\big|\leq& \big|E_{rel}[\partial^*U,\Sigma;(f_i-f)\chi_{\{y\geq y_2\}}]\big|\\
    &+C(|E_{rel}[\partial^*U,\Sigma]-E_{rel}[\partial^*U,\Sigma;\chi_{\{y\geq y_2\}}]|+y_2).
\end{align*}
For any $\epsilon'>0$, we fix $y_2<\min\{\epsilon,\epsilon'\}$ so small that $|E_{rel}[\partial^*U,\Sigma]-E_{rel}[\partial^*U,\Sigma;\chi_{\{y\geq y_2\}}]|<\epsilon'$. By the dominated convergence theorem, for any fix $y_2$, $\big|E_{rel}[\partial^*U,\Sigma;(f_i-f)\chi_{\{y\geq y_2\}}]\big|<\epsilon'$ for $i$ sufficiently large. So, putting these together, we see that for any $\epsilon'>0$, $\big|E_{rel}[\partial^*U,\Sigma;(f_i-f)]\big|\leq C\epsilon'$ for $i$ sufficienly large. This concludes the proof.
\end{proof}

\appendix\section{A computation for geodesics in hyperbolic place\label{64}}
We verify the relative entropy is conformally invariant for $n=1$ by computing out the explicit formula of the relative entropy for geodesics in hyperbolic plane. Specifically, consider the upper half space model, ($\mathbb{R}^2_+,\frac{1}{y^2}(dx\otimes dx+dy\otimes dy)$), of $\mathbb{H}^2$. Let $a_1<a_2<a_3<a_4$ be four points on the real line. Let $\gamma_1$ be the union of the geodesic joining $a_1$ and $a_2$ (denoted by $\gamma_{1,1}$) and the geodesic joining $a_3$ and $a_4$ (denoted by $\gamma_{1,2}$). Let $\gamma_2$ be the union of the geodesic joining $a_1$ and $a_3$ (denoted by $\gamma_{2,1}$) and the geodesic joining $a_2$ and $a_4$ (denoted by $\gamma_{2,2}$). We will give an explicit formula for $E_{rel}[\gamma_1,\gamma_2]$.

It is easy to check that $\gamma_{1,1}$ is the upper half circle defined implicitly by the equation
\begin{equation}\label{65}
    y^2+(x-a_1)(x-a_2)=0,\ y\geq0.
\end{equation}
Then,
\begin{align*}
    \mathcal{H}^1_\mathbb{H}(\gamma_{1,1}\cap\{y\geq\epsilon\})=\int_{\gamma_{1,1}\cap\{y\geq\epsilon\}}\frac{1}{y}d\mathcal{H}^1_\mathbb{H}=\int_{x_1(\epsilon)}^{x_2(\epsilon)}\frac{1}{y}\sqrt{1+\big(y'(x)\big)^2}dx,
\end{align*}
where $y$ is thought of as a function in $x$ and $x_1(\epsilon)<x_2(\epsilon)$ are the roots to the equation $\epsilon^2+(x-a_1)(x-a_2)=0$. Differentiating (\ref{65}), we get $yy'+(x-\frac{a_1+a_2}{2})=0$. Hence,
\begin{equation*}
    \sqrt{1+\big(y'(x)\big)^2}=\frac{1}{y}\sqrt{y^2+\left(x-\frac{a_1+a_2}{2}\right)^2}=\frac{a_2-a_1}{2y}.
\end{equation*}
As a result,
\begin{align*}
    \mathcal{H}^1_\mathbb{H}(\gamma_{1,1}\cap\{y\geq\epsilon\})=\int_{x_1(\epsilon)}^{x_2(\epsilon)}\frac{a_2-a_1}{2y^2}dx=\int_{x_1(\epsilon)}^{x_2(\epsilon)}\frac{a_2-a_1}{2(a_2-x)(x-a_1)}dx.
\end{align*}
Notice that $x_i(\epsilon)=\frac{1}{2}(a_1+a_2\pm\sqrt{(a_2-a_1)^2-4\epsilon^2})$. So,
\begin{equation*}
    \mathcal{H}^1_\mathbb{H}(\gamma_{1,1}\cap\{y\geq\epsilon\})=\ln\left(\frac{x_2(\epsilon)-a_1}{x_1(\epsilon)-a_1}\right)=2\ln\left(\frac{a_2-a_1+\sqrt{(a_2-a_1)^2-4\epsilon^2}}{2\epsilon}\right).
\end{equation*}
Then, we do similar computations for $\gamma_{1,2}$, $\gamma_{2,1}$ and $\gamma_{2,2}$. As explained in Section 4, we can pick a boundary defining function $r_1$ satisfying $r_1(\textbf{x},y)=y$ in a compact subset of $\overbar{\mathbb{R}^{n+1}_+}$ containing $\Omega$. So,
\begin{align*}
    E_{rel}[\gamma_1,\gamma_2]&=\lim_{\epsilon\to0}\sum_{i=1}^2\mathcal{H}^1_\mathbb{H}(\gamma_{1,i}\cap\{y\geq\epsilon\})-\sum_{i=1}^2\mathcal{H}^1_\mathbb{H}(\gamma_{2,i}\cap\{y\geq\epsilon\})\\
    &=\lim_{\epsilon\to0}2\ln\left(\frac{(a_2-a_1+\sqrt{(a_2-a_1)^2-4\epsilon^2})(a_4-a_3+\sqrt{(a_4-a_3)^2-4\epsilon^2})}{(a_3-a_1+\sqrt{(a_3-a_1)^2-4\epsilon^2})(a_4-a_2+\sqrt{(a_4-a_2)^2-4\epsilon^2})}\right)\\
    &=2\ln\left(\frac{(a_2-a_1)(a_4-a_3)}{(a_3-a_1)(a_4-a_2)}\right)=2\ln [a_1,a_4;a_2,a_3],
\end{align*}
where $[a_1,a_4;a_2,a_3]$ is the cross ratio. It is preserved under the M\"obius transformations of $\mathbb{S}^1$.

Let $\gamma_{3,1}$ be geodesic joining $a_1$ and $a_3$, and let $\gamma_{3,2}$ be the geodesic joining $a_2$ and $a_4$. Let $\gamma_3=\gamma_{3,1}\sqcup\gamma_{3,2}$. Then, $\gamma_3$ is stationary in $\mathbb{H}^2$ but has a singular point. Consider $E_{rel}[\gamma_i,\gamma_j]$ with $1\leq i,j\leq3$ and $i\neq j$. This gives $6$ different values of $E_{rel}$ in general. So, the relative entropy gives every permutation of the cross ratio of $a_1$, $a_2$, $a_3$ and $a_4$.

\bibliographystyle{alpha}
\bibliography{hyperbolic.bib}

\begin{thebibliography}{MW13}

\bibitem[AM10]{AlexakisMazzeo}
Spyridon Alexakis and Rafe Mazzeo.
\newblock Renormalized area and properly embedded minimal surfaces in
  hyperbolic 3-manifolds.
\newblock {\em Communications in Mathematical Physics}, 297(3):621--651, 2010.

\bibitem[And82]{Anderson1982}
Michael~T. Anderson.
\newblock Complete minimal varieties in hyperbolic space.
\newblock {\em Inventiones mathematicae}, 69(3):477--494, Oct 1982.

\bibitem[And83]{Anderson1983}
Michael~T. Anderson.
\newblock Complete minimal hypersurfaces in hyperbolic n-manifolds.
\newblock {\em Commentarii Mathematici Helvetici}, 58(1):264--290, Dec 1983.

\bibitem[Ber22]{BernsteinIsoper}
Jacob Bernstein.
\newblock A sharp isoperimetric property of the renormalized area of a minimal
  surface in hyperbolic space.
\newblock {\em Proc. Amer. Math. Soc.}, 2022.

\bibitem[BW21]{BWSpace}
Jacob Bernstein and Lu~Wang.
\newblock The space of asymptotically conical self-expanders of mean curvature
  flow.
\newblock {\em Math. Ann.}, 380:175--230, 2021.

\bibitem[BW22]{BWRelativeEntropy}
Jacob Bernstein and Lu~Wang.
\newblock Relative expander entropy in the presence of a two-sided obstacle and
  applications.
\newblock {\em Advances in Mathematics}, 399:108284, 2022.

\bibitem[Cos06]{Coskunuzer2006Generic}
Baris Coskunuzer.
\newblock Generic uniqueness of least area planes in hyperbolic space.
\newblock {\em Geometry \& Topology}, 10:401--412, 2006.

\bibitem[Cos11]{Coskunuzer2011Number}
Baris Coskunuzer.
\newblock On the number of solutions to the asymptotic plateau problem.
\newblock {\em Journal of G{\"o}kova Geometry Topology}, 5:1--19, 2011.

\bibitem[CY07]{ChenPseudolcality}
Binglong Chen and Le~Yin.
\newblock Uniqueness and pseudolocality theorems of the mean curvature flow.
\newblock {\em Communications in Analysis and Geometry}, 15:435--490, 2007.

\bibitem[dOS98]{deOliveiraSoret98}
Geraldo de~Oliveira and Marc Soret.
\newblock Complete minimal surfaces in hyperbolic space.
\newblock {\em Mathematische Annalen}, 311(3):397--419, Jul 1998.

\bibitem[Gab97]{Gabai97}
David Gabai.
\newblock On the geometric and topological rigidity of hyperbolic 3-manifolds.
\newblock {\em Journal of the American Mathematical Society}, 10(1):37--74,
  1997.

\bibitem[Giu84]{GiustiBook}
Enrico Giusti.
\newblock {\em Minimal Surfaces and Functions of Bounded Variation}.
\newblock Monographs in Mathematics. Birkh{\"a}user Boston, 1984.

\bibitem[HL87]{HardtLin}
Robert Hardt and Fang-Hua Lin.
\newblock Regularity at infinity for area-minimizing hypersurfaces in
  hyperbolic space.
\newblock {\em Inventiones mathematicae}, 88(1):217--224, 1987.

\bibitem[HW15]{HuangWang2015Counting}
Zheng Huang and Biao Wang.
\newblock Counting minimal surfaces in quasi-fuchsian three-manifolds.
\newblock {\em Transactions of the American Mathematical Society},
  367(9):6063--6083, 2015.

\bibitem[Lin89]{FLin}
Fang-Hua Lin.
\newblock On the dirichlet problem for minimal graphs in hyperbolic space.
\newblock {\em Inventiones mathematicae}, 96(3):593--612, 1989.

\bibitem[LSU88]{LadyzenskajaLinearQuasil}
O.A. Lady{\v{z}}enskaja, V.A. Solonnikov, and N.N. Ural'ceva.
\newblock {\em Linear and Quasi-linear Equations of Parabolic Type}.
\newblock Translations of mathematical monographs. American Mathematical
  Society, 1988.

\bibitem[MW13]{MartinWhite2013}
Francisco Martin and Brian White.
\newblock Properly embedded, area-minimizing surfaces in hyperbolic 3-space.
\newblock {\em arXiv: Differential Geometry}, 2013.

\bibitem[Ngu21]{NguyenWeightedMon}
Manh~Tien Nguyen.
\newblock {Weighted monotonicity theorems and applications to minimal surfaces
  in hyperbolic space}.
\newblock \url{https://arxiv.org/abs/2105.12625}, 2021.
\newblock Preprint.

\bibitem[RW99]{GrahamWitten}
C.~{Robin Graham} and Edward Witten.
\newblock Conformal anomaly of submanifold observables in ads/cft
  correspondence.
\newblock {\em Nuclear Physics B}, 546(1):52--64, 1999.

\bibitem[Ton96]{TonegawaCMCinHpSpace}
Yoshihiro Tonegawa.
\newblock Existence and regularity of constant mean curvature hypersurfaces in
  hyperbolic space.
\newblock {\em Mathematische Zeitschrift}, 221(1):591--615, 1996.

\bibitem[Tyr22]{TyrrellRenADim4}
Aaron~J. Tyrrell.
\newblock {Renormalized Area for Minimal Hypersurfaces of 5D
  Poincar\'e-Einstein Spaces}.
\newblock \url{https://arxiv.org/abs/2204.06663}, 2022.
\newblock Preprint.

\bibitem[Yao]{YaoMountainPass}
Junfu Yao.
\newblock A mountain-pass theorem in hyperbolic space and its application.
\newblock In preparation.

\end{thebibliography}


\begin{thebibliography}{20}
%\bibitem{1}
        %S.B. Angenent, T. Ilmanen, and D.L. Chopp, \textit{A computed example of non-uniqueness of mean curvature flow in $\mathbb{R}^3$}, Commun. in Partial Differential Equations 20 (1995), no. 11-12, 1937-1958.
\bibitem{1}Alexakis, S., Mazzeo, R. Renormalized Area and Properly Embedded Minimal Surfaces in Hyperbolic 3-Manifolds. Commun. Math. Phys. 297, 621–651 (2010). https://doi.org/10.1007/s00220-010-1054-3


\bibitem{2}Bernstein, J., Wang, L. The space of asymptotically conical self-expanders of mean curvature flow. Math. Ann. 380, 175–230 (2021). https://doi.org/10.1007/s00208-021-02147-0


\bibitem{3}Bernstein, J., Wang, L. Relative expander entropy in the presence of a two-sided obstacle and applications. preprint (2019). https://arxiv.org/abs/1906.07863


\bibitem{4}Graham C. R., Witten, E. Conformal anomaly of submanifold observables in AdS/CFT correspondence. Nuclear Physics B, Volume 546, Issues 1–2,
1999,
Pages 52-64,
ISSN 0550-3213,
https://doi.org/10.1016/S0550-3213(99)00055-3.


\end{thebibliography}

\iffalse\fi
\end{document}